\documentclass[12pt]{article}
\usepackage{amsmath, amssymb, amsfonts,amsthm}
\usepackage[english]{babel}
\usepackage{algorithm2e}
\usepackage{hyperref}
\usepackage{enumitem}
\usepackage{graphicx}
\usepackage{caption}
\usepackage{xcolor}
\usepackage{babel}
\usepackage{comment}
\usepackage[normalem]{ulem}
\usepackage{tikz}
\usetikzlibrary{shapes, arrows, positioning,patterns}
\usepackage{xstring}
\usetikzlibrary{calc}

\theoremstyle{theorem}
\newtheorem{theorem}{Theorem}[section]
\newtheorem{lemma}[theorem]{Lemma}
\newtheorem{corollary}[theorem]{Corollary}
\newtheorem{conjecture}[theorem]{Conjecture}
\newtheorem{proposition}[theorem]{Proposition}
\newtheorem{claim}[theorem]{Claim}

\theoremstyle{definition}
\newtheorem{remark}[theorem]{Remark}

\title{Weak saturation numbers of large complete
bipartite graphs}

\date{}
\begin{document}

\author{Margarita Akhmejanova\footnote{Umeå universitet, margarita.akhmejanova@umu.se}, \, Ilya Vorobyev\footnote{
IOTA Foundation, Berlin, Germany, ilia.vorobev@iota.org}, \, Maksim Zhukovskii \footnote{School of Computer Science, University of Sheffield, UK, m.zhukovskii@sheffield.ac.uk
}}

\maketitle

\textbf{Abstract.}  An $n$-vertex graph $G$ is weakly $F$-saturated if $G$ contains no copy of \( F \) and there exists an ordering of all edges in \( E(K_n) \setminus E(G) \) such that, when added one at a time, each edge creates a new copy of \( F \). The minimum size of a weakly \( F \)-saturated graph \( G \) is called the weak saturation number $\mathrm{wsat}(n, F)$. We obtain exact values and new bounds for $\mathrm{wsat}(n, K_{s,t})$ in the previously unaddressed range $s+t < n < 3t-3$, where $3\leq s\leq t$. To prove lower bounds, we introduce a new method that takes into account connectivity properties of subgraphs of a complement $G'$ to a weakly saturated graph $G$. We construct an auxiliary hypergraph and show that a linear combination of its parameters always increases in the process of the deletion of edges of $G'$. This gives a lower bound which is tight, up to an additive constant.


\section{Introduction}

Let \(F\) be a fixed graph. An $n$-vertex graph \(G\) is said to be \emph{weakly \(F\)-saturated} if it is \(F\)-free and there exists an ordering \(e_1, e_2, \ldots, e_m\) of the edges in \(E(K_n) \setminus E(G)\) such that, for each \(i \in \{1, \ldots, m\}\), the graph \(G \cup \{e_1, \ldots, e_i\}\) contains a copy \(F_i\) of \(F\) with \(e_i \in E(F_i)\). The \emph{weak saturation number} \(\mathrm{wsat}(n,F)\) is the minimum number of edges in a weakly \(F\)-saturated graph. This notion was introduced by Bollob\'{a}s~\cite{Bol86}, motivated by a related notion of (strong) saturation in hypergraphs. Bollob\'{a}s conjectured that for \( s \geq 2 \),
\begin{equation}
\mathrm{wsat}(n, K_s) = \binom{n}{2} - \binom{n - s + 2}{2},
\label{Bol:conjecture}
\end{equation}
with equality achieved for a graph $G$ obtained from $K_n$ by excluding a clique \( K_{n-s+2} \). The conjecture was proved independently by Lov\'{a}sz \cite{Lovasz77}, Frankl \cite{FRANKL1982125}, Alon \cite{ALON198582}, and Kalai \cite{Kalai}. Notably, a similar conjecture for strongly saturated graphs was proved by Erd\H{o}s, Hajnal, and Moon \cite{Erdos-Hajnal-Moon}, and later generalised by Bollob\'{a}s to hypergraphs \cite{Bollobs1965OnGG}. In~\cite{Kalai}, Kalai proved~\eqref{Bol:conjecture} via a general linear algebraic approach which was further applied in many other scenarios (see an overview of this approach and its applications, e.g., in~\cite{TZ:combi}) --- in particular, for complete bipartite graphs $F$, as we discuss below.

A significant number of papers focus on the weak saturation number for bipartite graphs $F$. Borowiecki and Sidorowicz~\cite{Borowiecki2002} proved that $\mathrm{wsat}(n, K_{1,t}) = \binom{t}{2}$ for $n > t+1$, and a shorter proof of this result was later given by Faudree, Gould, and Jacobson~\cite{Faudree2013}. Moreover, Borowiecki and Sidorowicz~\cite{Borowiecki2002} proved that $\mathrm{wsat}(n, K_{2,2}) = n$ for all $n > 4$, and Faudree, Gould, and Jacobson~\cite{Faudree2013} established $\mathrm{wsat}(n, K_{2,3}) = n+1$ for all $n > 5$. Fairly recently, Miralaei, Mohammadian, and Tayfeh-Rezaie~\cite{Miralaei2023} determined $\mathrm{wsat}(n, K_{2,t})$ for all $n > t+2$: they proved that, for all $t\geq 3$, if $n\geq 2t-1$ or $t+2<n\leq 2t-2$ and $t$ is odd, then $\mathrm{wsat}(n, K_{2,t})=n-2+{t\choose 2}$ and, if $t+2<n\leq 2t-2$ and $t$ is even, then $\mathrm{wsat}(n, K_{2,t})=n-1+{t\choose 2}$. This completed the investigation of $\mathrm{wsat}(n,K_{s,t})$ for all $s\leq 2$. We note that obtaining the {\it exact} value of $\mathrm{wsat}(n,K_{2,t})$ is much more challenging than determining its asymptotics, up to a constant additive term. Indeed, it is known that $\mathrm{wsat}(n,F)=n+\Theta(1)$ for all connected graphs $F$ with minimum degree 2, see~Claim 1.3~in~\cite{TZ:combi} and the discussion in front of this claim.

Using symmetric algebras, Kalai~\cite{Kalai} found the weak saturation number for balanced bipartite graphs $F$ for all large enough $n$: if \( s \geq 2 \) and \( n \geq 4s - 4 \),
$$
\mathrm{wsat}(n, K_{s, s}) = (s-1)\left(n + 1 - \frac{s}{2}\right).
$$
Later, using a similar approach, Kronenberg, Martins, and Morris \cite{Kronenberg_2021} extended this result to \( s \geq 2 \) and \( n \geq 3s - 3 \). Additionally they showed that for these values of \( s \),
$$
\mathrm{wsat}(n, K_{s, s+1}) = (s-1)\left(n + 1 - \frac{s}{2}\right) + 1.
$$
Also, for \( 2\leq s < t \), they estimated \( \mathrm{wsat}(n, K_{s, t}) \) up to an additive constant $C=C(s,t)$ that does not depend on $n$:
\begin{align*}%\label{Martins2}
 \mathrm{wsat}(n, K_{s, t})\leq (s-1)(n - s) + \frac{t(t-1)}{2} & \quad \text{ if }n\geq 2(s+t)-3; \\
 \mathrm{wsat}(n, K_{s, t})\geq (s-1)(n - t + 1) + \frac{t(t-1)}{2} &\quad\text{ if }n\geq 3t-3.
\end{align*}

In this paper, we focus on the case of small $n$ --- namely, $s+t\leq n< 3t-3$, which is almost entirely open. The only exception is the following recent result of Miralaei, Mohammadian, and Tayfeh-Rezaie~\cite{Miralaei2023} resolving the case $n=s+t$:
$$
\mathrm{wsat}(s+t, K_{s,t}) =
\begin{cases}
\binom{s+t-1}{2}, & \text{if }\gcd(s,t)=1,\\[6pt]
\binom{s+t-1}{2}+1, & \text{otherwise}.
\end{cases}
$$
%thereby correcting an earlier claim in the literature.

The intermediate range \( s + t < n < 3t-3 \) appears to be substantially more challenging. In particular, the construction underlying the upper bound in the case of large $n$ no longer applies and the answer becomes significantly different from \((s-1)n\). In such a case, the linear-algebraic method seems to be inefficient --- in particular, the optimal bound for the dimension of a suitable vector space is supposedly too large to apply the general framework explained in~\cite[Theorem 2.6]{TZ:matroids}. We also note that in~\cite{TZ:matroids} it was proved that, for certain graphs $F$, the best possible lower bound achieved by the linear algebraic approach is strictly less than $\mathrm{wsat}(n,F)$. We expect it is also the case here, although we fail to prove it.

We develop a new combinatorial framework that allows us to solve the case \( n = s + t + 1 \) and to obtain new non-trivial bounds for larger \( n \).

\begin{theorem}\label{result_1}
For all \( s > 2 \),
\[
\mathrm{wsat}(2s+1, K_{s, s}) = \binom{2s+1}{2} - (4s - 4).
\]
\end{theorem}

When $s\neq t$ and $\mathrm{gcd}(s,t)\neq 1$, our method yields bounds that differ by 1. Although we expect a more involved analysis might give the tight answer in this case, the main purpose of this paper is to introduce the new method. Thus, we have decided to present a more transparent proof in this paper and defer the technical details to future work.

\begin{theorem}\label{result_2}
For integers \( t, s \) such that \(t > s > 2 \),
\begin{align*}
\mathrm{wsat}(s+t+1, K_{s, t})=\binom{s + t + 1}{2} - (2s + 2t - 2), &\quad\text{ if }\mathrm{gcd}(s,t)=1;\\
2s + 2t - 3\leq\binom{s + t + 1}{2}-\mathrm{wsat}(s+t+1, K_{s, t})\leq 2s + 2t - 2,
&\quad\text{ otherwise.}
\end{align*}
\end{theorem}

We also show that the method can be generalised to larger $n$. In particular, we applied its simplified version to get the last bound in the theorem that follows for $\mathrm{wsat}(s+t+j,K_{s,t})$, where $j=2$. Although this most straightforward implementation gives suboptimal bounds for larger $j$, we believe that an advancement of our approach that takes into account high connectivity properties of subgraphs of complements to weakly saturated graphs may give much tighter results --- see discussions in Section~\ref{sc:further}. 

Actually we show that known bounds on the maximum number of edges in a graph without $k$-connected subgraphs yield reasonable lower bounds on the weak saturation number, identifying the order of magnitude of the second order term:
$$
 \mathrm{wsat}(s+t+j,K_{s,t})={s+t+j\choose 2}-\Theta(j(s+t)).
$$
In particular, for $s=t$, we get that the difference between the bounds is $(7/6+o_j(1))js$. Below we state our result for arbitrary $j$. Note that our upper bound coincides, up to a constant additive term, with the bounds from~\cite{Kronenberg_2021} for those pairs $(t,s)$, where both results are valid, and we conjecture that this upper bound is essentially tight --- see Section~\ref{sc:further}.

%The next theorem give new lower and upper bounds on $\mathrm{wsat}(n,K_{s,t})$ when the difference $j:=n-(s+t)$ is bigger than 1. The gap between the bounds is $(1/2+o_n(1))nj^2$ [not clear what it means... maybe present in a different form...]. [Comparison with bound in Kronenberg et al] When $s=t$ and $j=t-3$ --- the only case when we may compare it with Kronenberg et al (?), we get the lower bound $(s-2)(2s-2)+2s-3$ and the upper bound $s((s-2)(s-1)+2)-1=\Theta(s^3)$. So, the upper bound only makes sense when $j=O(1)$. Can we get a better upper bound for large $j$? At least when $s=t$? Let $s=t$, $j=2$. Then the lower bound is $2(2s-2)+2s-3=6s-7$ and the upper bound is $8s-1$, so the difference is $n+4$. Can we make it $o(n)$? Or $O(1)$?

%\begin{theorem}\label{result_3}
%$|E(K_{2s+j})|-2s\left(\frac{j(j+1)}{2}+1\right)\leq \mathrm{wsat}(K_n, K_{2s+j})\leq |E(K_{2s+j})|-2s(j + 1)+2j+3.$
%\end{theorem}

\begin{theorem}\label{result_3}
For all $2 < s \leq t$ and $2 \leq j < t - 2$,
\begin{align}
 \label{eq:th3-1}
 \mathrm{wsat}(s+t+j, K_{s, t})\leq {s+t+j\choose 2}-j(s + t - 2) - 2t + 3,&\quad \text{ if }2 \leq j < t - 2;\\
 \label{eq:th3-2}
 \mathrm{wsat}(s+t+j, K_{s, t})\geq {s+t+j\choose 2}-\frac{19}{12}(j+1)(s+t-1),&\quad \text{ if }3 \leq j \leq \frac{2}{3}(s+t)-\frac{5}{3};\\
  \label{eq:th3-3}
  \mathrm{wsat}(s+t+2, K_{s, t})\geq {s+t+2\choose 2}-4(s + t)+ 1.
\end{align}
\end{theorem}

%[Same bounds on $j$ for the lower and for the upper bound?]

As we have already mentioned,~\eqref{eq:th3-2} follows from the fact that the number of edges $\varphi_n(k)$ in an $n$-vertex graph that does not have $(k+1)$-connected subgraphs is bounded by $O(kn)$. The classical result of Mader~\cite{Mader} asserts that $\varphi_n(k)\leq(1+1/\sqrt{2}-o(1))kn$, for every $k\geq 2$. Mader conjectured that the right bound is $3/2(k-1/3)(n-k)$ and proved that it cannot be less --- there are arbitrarily large graphs with exactly $3/2(k-1/3)(n-k)$ edges that do not have $(k+1)$-connected subgraphs. Yuster~\cite{Yuster} improved the constant in front of $kn$ to $193/120$, and then Bernshteyn and Kostochka~\cite{BK} got the current record $19/12$. Mader also verified the conjecture for all $k\in[2,6]$ and large enough $n$. We note that it implies the following bound for $j=2$ and {\it large enough} $s+t$:
$$
  \mathrm{wsat}(s+t+2, K_{s, t})\geq {s+t+2\choose 2}-4(s + t)+4.
$$
Although this bound is slightly better than~\eqref{eq:th3-3}, we keep the weaker version since it holds {\it for all $s$ and $t$}. We also believe that the observation that a complement to a weakly $K_{s,t}$-saturated graph on $s+t+j$ does not contain $(j+2)$-connected subgraphs is not enough to estimate the weak saturation number precisely --- see the discussion in Section~\ref{sc:further}.

\begin{remark}
For $j\in\{3,4,5\}$ and large enough $s+t$, a better lower bound follows from Mader's result: 
$$
 \mathrm{wsat}(s+t+j, K_{s, t})\geq {s+t+j\choose 2}-\frac{3j+2}{2}(s+t-1).
$$
\end{remark}

\paragraph{The method: hyperforests and semi-invariants.} Let $G$ be a weakly $K_{s,t}$-saturated graph on $s+t+j$ vertices and let $G'$ be its edge-complement. Proofs of lower bounds in Theorems~\ref{result_1}--\ref{result_3} are based on the observation that $G'$ does not contain $(j+2)$-connected subgraphs (see Claim~\ref{cl:connectivity_reduction}) and each step in a sequential deletion of edges in $G'$, that corresponds to the addition of edges to $G$, breaks some $k$-connected component for $k\leq j+1$. 

For $j=1$, the maximum number of edges in an $n$-vertex graph without 3-connected subgraphs equals $\frac{5}{2}(n-2)$, due to Mader~\cite{Mader}. This gives the bound $\mathrm{wsat}(s+t+1,K_{s,t})\geq{s+t+1\choose 2}-\frac{5}{2}(s+t-1)$ which is $\left(\frac{1}{2}(s+t)+O(1)\right)$-far from the answer. Therefore, it is essential to take into account the speed of destruction of connected components. We model the process of destruction by a sequence of hyperforests on $V(G')$, whose hyperedges do not intersect non-trivially with 2-connected components. Initially each hyperedge coincides with a 1-connected component of $G'$, and at the end of the process all vertices become isolated. Denoting by $f_i$ and $c_i$ the number of hyperedges and components of the $i$-th hyperforest, we prove that $f_i+2c_i\geq f_{i-1}+2c_{i-1}+1$, which gives the required control of the speed. This approach leads to a lower bound on $\mathrm{wsat}(s+t+1,K_{s,t})$ which is at most 3-far from the answer. Some technical case analysis is required to show that there are always few expensive steps, where the difference $f_i+2c_i-(f_{i-1}+2c_{i-1}+1)$ is positive.

Getting tight bounds for larger $j$ is challenging since it requires introducing more complex combinatorial structures and controlling more parameters. In order to put the first brick in the wall and to demonstrate the efficiency of our approach, in the current paper we restrict ourselves with two simple applications of our ideas. First, we show that in the case $j=2$, counting the number of hyperedges in the auxiliary hypergraph is enough to get a lower bound which is as good as the application of Mader's result (and allows to get the bound for small $s+t$, where the result of Mader is not applicable). Second, for $j\geq 3$, we derive our bounds from the improvement of Mader's result due to Bernshteyn and Kostochka~\cite{BK}. We believe that a combination of these two arguments in the spirit of our proof in the case $j=1$ could at least lead to a tight estimation of the second-order term.

\paragraph{Organisation.} In Section~\ref{sc:pre} we define erase processes and auxiliary hyperforests and prove their properties that we use for the lower bounds in Theorems~\ref{result_1}~and~\ref{result_2}. These theorems are proved in Sections~\ref{sc:1_proof}~and~\ref{sc:2_proof}, respectively. The proof of Theorem~\ref{result_3} is presented in Section~\ref{sc:3_proof} and requires a modified erase procedure, that we define at the beginning of Section~\ref{sc:2_lower_proof}. Finally, in Section~\ref{sc:further} we discuss some open questions.

%\subsection{Organization}

%The next part of our Introduction introduces our notation and proves the upper bound of Theorem \ref{result_1}. The main text of the article begins with a section on the core arguments for Theorems \ref{result_1} and \ref{result_2}, focusing on the concepts of hypertrees and semi-invariants. This section provides a theoretical foundation, including definitions and properties of these key concepts. The following section presents proofs of Theorems \ref{result_1} and \ref{result_2}, with particular emphasis on the upper bound of Theorem \ref{result_2}, illustrated through graph construction. The lower bound of Theorem \ref{result_2} is also addressed. We then expand the arguments to the case where \( n = s + t + j \), discussing the lower and upper bounds of Theorem \ref{result_3} in this context. The article concludes with a section on open problems.

\section{Erase process and auxuliary hyperforests}
\label{sc:pre}

To obtain bounds for $\mathrm{wsat}(s+t+1, K_{s,t})$, we introduce an equivalent {\it erase process} in the complement graph: at each step, a certain edge is erased under constraints corresponding to the creation of a new $K_{s,t}$. This viewpoint allows us to study maximal erasable graphs instead of weakly saturated ones, and to describe their structure through auxiliary operations and associated hyperforests, which will be central in our arguments.

In Section~\ref{sc:erase} we define the erase process and show that the maximum number of edges in an erasable graph and the weak saturation number sum up to the number of edges in a clique ${s+t+1\choose 2}$. We will prove both upper and lower bounds in Theorems~\ref{result_1}~and~\ref{result_2} in terms of erasable graphs. The main novel ingredient --- the auxiliary hyperforests and an associated semi-invariant --- will be introduced in Sections~\ref{sc:hyperforests}~and~\ref{sc:invariants}. 

\subsection{Erase process}
\label{sc:erase}

Let \( G = (V, E) \) be a graph with \( s + t + 1 \) vertices and  \( (V_1, V_2, v) \) be a partition of \( V \) such that:
\begin{enumerate}
    \item  \( |V_1| = s \) and \( |V_2| = t \),
    \item there is exactly one edge \( e \) between \( V_1 \) and \( V_2 \).
\end{enumerate}
Define \( \text{Erase}(e, G) \) as the procedure that removes edge \( e \) from \( G \). An edge \( e \in E(G) \) is called \textit{erasable} in \( G \) if there exists a partition \( (V_1, V_2, v) \) of \( V \) that permits the application of \( \text{Erase}(e, G) \), and we call the witness vertex $v$ {\it closed}. A graph \( G \) with \( |E(G)| = m \) is called \textit{erasable} if there exists an ordering \( e_m, \ldots, e_1 \) of \( E(G) \) such that for each \( i \in [m] \), the edge \( e_i \) is erasable in \( G \setminus \{e_m, \ldots, e_{i+1}\} \). The sequence of graphs $G_0=G$, $G_1=G_0\setminus e_1$, $G_2=G_1\setminus e_2,$ $\ldots,$ $G_m=\varnothing$ is called an {\it erase process}.

\begin{claim}\label{claim}
Let \( G \) be an erasable graph with the maximum number of edges among all erasable graphs on \( s + t + 1 \) vertices. Then,
\[
|E(G)| = \binom{s + t + 1}{2} - \mathrm{wsat}(s + t + 1, K_{s, t}).
\]
\end{claim}

\begin{proof}
Let $G$ be a graph on $s+t+1$ vertices and $\overline{G}$ be its edge-complement. It is sufficient to show that $G$ is erasable if and only if $\overline{G}$ is weakly $K_{s,t}$-saturated. To see this, we couple the process of weak saturation with the erase procedure: each time we add $e_i$ to $\overline{G}_i$ together with a copy of $K_{s,t}$ and get $\overline{G}_{i+1}$, we observe that the parts $V_1,V_2$ of this copy and the remaining vertex $v$ witness the fact that $e_i$ is erasable in $G_i$ (and vice versa).
\end{proof}

%It will be more convenient for us to estimate the maximum number of edges in an erasable graph instead of analysing the weak saturation number directly.\\

Let us now fix a graph \( G \) with the maximum number of edges among all erasable graphs on $s+t+1$ vertices, and let \( m = |E(G)| \). Fix an erase process \( \mathcal{P} \). Enumerate the edges in \( G \) according to the order in which they are erased in the process \( \mathcal{P} \): $e_1$ is erased first and $e_m$ is erased last. Let \( G_0 = G, G_1, G_2, \ldots, G_m=\varnothing \) be the respective sequence of graphs in accordance with \( \mathcal{P} \). For each erased $e_i$, we fix its witness partition $(V_i^1,V_i^2,\{v_i\})$.

\subsection{Construction of the hyperforest and its properties}
\label{sc:hyperforests}

To each graph \( G_i = G \setminus \{e_1, \ldots, e_{i-1}\},~ i \in [m] \), we assign a hyperforest\footnote{A \textit{hyperforest}  is a hypergraph \( H = (V, E) \) that contains no cycles. A cycle in a hypergraph is a sequence of hyperedges \( e_1, e_2, \dots, e_k \in E \) and vertices \( v_1, v_2, \dots, v_k \in V \), where each vertex belongs to two consecutive hyperedges, forming a closed loop: $v_1\in e_1\cap e_2$, $v_2\in e_2\cap e_3$, $\ldots,$ $v_k\in e_k\cap e_1$. In particular, substituting $k=2$, we get that a hyperforest do not contain a pair of hyperedges that share at least 2 vertices.} \( H_i \) (its construction is presented in what follows), built on the same set of vertices and satisfying the following properties:
\begin{enumerate}\label{three properties}
\item for each edge \( e \in E(G_i) \), there exists a hyperedge \( F \in H_i \), such that \( e \subset F\),
\item for any hyperedge \( F \in E(H_i) \), the induced subgraph \( G_i[F] \) is connected.
\end{enumerate}

We start from the initial graph $G_0$ and assign to it a hypergraph $H_0$ in the following way: each hyperedge of $H_0$ consists of vertices of a connected component of $G_0$. We then inductively construct \( H_i, i \in [m] \), via the following operations. % that produce from a hyperedge $F$ of $H_{i-1}$ several hyperedges of $H_i$.

\begin{itemize}

\item \textit{Edge-operation} is applied to a hyperedge $F$ of $H_{i-1}$ such that $e_i\subset F$. Assuming that $H_{i-1}$ is a hyperforest satisfying the property 1, there is an exactly one such hyperdge $F$. We then replace $F$ with the set of hyperedges that form connected components of $G_{i-1}[F] \setminus e$ of size more than 1. In particular, if $G_{i-1}[F] \setminus e$ is connected, then $F$ remains unchanged. Note that the edge-operation replaces $F$ with zero, one, or two hyperedges. Denote by $O(e)$ the set of new hyperedges.

\item \textit{Vertex-operation} is applied to all hyperedges $F$ that contain the closed vertex $v_i$. Let \( F_1, F_2, \ldots, F_h \) be the connected components of \( G_i[F\setminus \{v_i\}] \). We replace \( F \) with hyperedges \( F_1 \cup \{v_i\}, F_2 \cup \{v_i\}, \ldots, F_h \cup \{v_i\} \). Denote by $O(F,v_i)$ the set of new hyperedges that this operation produces from $F$.

\end{itemize}

For every $i\in[m]$, the hypergraph $H_i$ is constructed in the following way: (1) find the unique $F\in E(H_{i-1})$ containing $e_i$ and replace $F$ with $O(e)$, denote the new intermediate hypergraph by $H$; (2) find the set $\mathcal{F}_i$ of all $F\in E(H)$ containing $v_i$ and, for every $F\in\mathcal{F}_i$, replace $F$ with $O(F,v_i)$.

\begin{claim}
For every $i\in\{0,1,\ldots,m\}$, $H_i$ is a hyperforest satisfying properties 1 and 2. % such that every edge of $G_i$ belongs to some hyperedge of $H_i$.
\end{claim}

\begin{proof}
We prove the claim by induction over $i$. The case $i=0$ is straightforward. 

Assume $H_{i-1}$ is a hyperforest covering all edges of $G_{i-1}$. Let $F\in E(H_{i-1})$ be the only hyperedge that contains $e_i$. The edge-operation partitions $F$ into subsets and thus it does not produce cycles. Assume some edge $e\in E(G_i)$ does not belong to a hyperedge of the new hypergraph. Then it belongs to $F$ and, therefore, it belongs to some connected component of $F[G_i\setminus e_i]$. Therefore, it must belong to a hyperedge of the new hypergraph --- a contradiction. The fact that all hyperedges of the new hyperforest induce connected graphs is straightforward.

It remains to prove that each vertex-operation preserves the property of being a hyperforest and preserves properties 1 and 2 --- keeps all edges of $G_i$ covered and all hyperedges connected. We will prove it, again, by induction, assuming that all the previous vertex-operations preserve the required properties. Since the current hypergraph $H$ (before applying the vertex operation to $F$) is a hyperforest, each edge from $O(F,v_i)$ has at most one common vertex with every hyperedge from $E(H)\setminus\{F\}$. Any two sets from $O(F,v_i)$ have one common vertex by definition. Therefore, the new hypergraph is indeed a hyperforest. Next, every edge $e\in E(G_i)$ is either inside a hyperedge from $E(H)\setminus \{F\}$, or it belongs to $F$. In the latter case, it cannot join two different sets $F_j,F_{j'}$. Therefore, it belongs to one of the sets $F_1\cup\{v_i\},\ldots,F_h\cup\{v_i\}$, as needed. It remains to show that all $G_i[F_j\cup\{v_i\}]$, $j\in[h]$, are connected. Graphs $G_i[F_j]$ are connected by the definition of a vertex-operation. If $v_i$ does not have $G_i$-neighbours in $F_j$, then $F_j$ is a connected component of $G_i[F]$ --- a contradiction with the induction assumption that $G_i[F]$ is connected.
\end{proof}

\subsection{Semi-invariant}
\label{sc:invariants}

Let \( f_i \) and \( c_i \) denote the number of hyperedges and connected components, respectively, in the auxiliary hyperforest \( H_i \), where $i\in[m]$. We note that each isolated vertex in this hyperforests counts as a connected component. 

We explore the dynamic of the vector \( (f_i, c_i) \) as \( i \) grows from \( 1 \) to \( m \). Our aim is to show that 
$$
 s_i:=f_i+2c_i
$$
strictly increases in $i$. Indeed, if, for every $i$, $s_i-s_{i-1}\geq 1$, then 
\begin{equation}
 2(s+t+1)-3\geq f_m+2c_m-(f_0+2c_0)=s_m-s_0=\sum_{i=1}^m(s_i-s_{i-1})\geq m,
 \label{eq:Q}
\end{equation}
giving a lower bound, which is 3-far from the lower bound in Theorem~\ref{result_1} and 1-far from the lower bound in Theorem~\ref{result_2}. Therefore, we will only need to show, that few steps $i$ are suboptimal with $s_i-s_{i-1}\geq 2$.

We first describe all possible changes of \((f_i,c_i)\) in one step.

\begin{lemma}\label{events}
For every \( i \in [m] \) there exists a nonnegative integer \( \lambda=\lambda(i) \) and a vector
\[
\overrightarrow{v} \in \{(1,0),\ (1,1),\ (0,1),\ (-1,1)\}
\]
such that
\[
(f_i, c_i) = (f_{i-1}, c_{i-1}) + \overrightarrow{v} + (\lambda,0).
\]
\end{lemma}

\begin{proof} Every step $i\in[m]$ consist of sequential applications of a single edge-operation and several vertex-operations. Assume that, after the edge-operation applied to a hyperedge $F$, the connected graph $G_{i-1}[F]$ splits into two connected components on sets of vertices $V_1,V_2$, not necessarily non-trivial. If one of these sets is empty, then $(f_{i-1},c_{i-1})$ remains unchanged. Otherwise, consider three cases: \( |V_1| > 1 \) and \( |V_2| > 1 \); \( |V_1| = 1 \) and \( |V_2| > 1 \); and \( |V_1| = 1 \) and \( |V_2| = 1 \).
\begin{enumerate}
\item if \( |V_1|>1 \) and \( |V_2|>1 \), then a new component emerges, giving the increment \((1,1)\);
\item if exactly one of \( |V_1|,|V_2| \) equals \(1\), the number of hyperedges remains unchanged while the number of components increases by one, giving the increment \((0,1)\);
\item if \( |V_1|=|V_2|=1 \), we lose one hyperedge and get a new connected component, giving the increment \((-1,1)\).
\end{enumerate}
Any subsequent vertex-operation does not change the number of connected components and does not decrease the number of hyperedges. Therefore, it only remains to prove that, if either $V_1=\varnothing$ or $V_2=\varnothing$, then the number of edges gets bigger. Indeed, if the edge-operation is applied to $F$ and $F$ does not split, then $v_i\in F$. Therefore, a vertex-operation is also applied to $F$ and then $F$ has to split in at least two hyperedges, as needed.
\end{proof}

We then immediately get the desired corollary.

\begin{corollary}
\label{S(H)}
For every \( i \in [m] \), $s_i\geq s_{i-1}+1$.
\end{corollary}
Let 
\[
Q_i := s_i - s_{i-1} - 1, \qquad Q := \sum_{i=1}^m Q_i.
\]
Here \(Q_i\) measures the excess growth of \(s_i\) beyond the minimal possible increase of \(1\) per step. To prove the lower bounds in Theorems~\ref{result_1}~and~\ref{result_2}, it suffices to show \(Q \ge 3\) and \(Q \ge 1\), respectively. Note that some values of the increment $\overrightarrow{v}$ from Lemma~\ref{events} increase \(s_i-s_{i-1}\) by more than \(1\): for instance, the increments \((0,1)\) and \((1,1)\) make \(Q_i\) equal to \(1\) and \(2\), respectively. The rest of the proof of the lower bounds is devoted to showing that if initially $f_0+2c_0\leq 5$ when $s=t$ or $f_0+2c_0\leq 3$ when $s\neq t$, then there is always a couple of such suboptimal steps that result in the desired lower bounds on the total excess $Q$.

% Note that if $(e_i,c_i) = (e_{i-1},c_{i-1}) + b(1,0) + a(0,1)$ for some $a,b \in \mathbb{Z}$, then $Q_i = 2a + b - 1$.

%\begin{note}\label{incr}
%Let \(B = A + (0,1)\) and \(C = A + (1,1)\). Then
%\[
%Q(B) \leq Q(A) + 1, 
%\qquad 
%Q(C) \leq Q(A) + 2.
%\]
%\end{note}

%Informally, Note~\ref{incr} states that such steps are \emph{non-optimal}, in the sense that the semi-invariant \(s_i\) increases by more than one.

\section{Proof of Theorem~\ref{result_1}}
\label{sc:1_proof}

\subsection{Upper bound}

To prove the upper bound, by Claim~\ref{claim} it suffices to construct an erasable graph with \(4s-4\) edges. Such a graph can be obtained from a disjoint \((s-1)\)-path \(P_1\), an \((s-2)\)-path \(P_2\), and a \(C_4\) cycle \((a,b,c,d)\) by joining \(a\) to all vertices of \(P_1\) and \(b\) to all vertices of \(P_2\). Figure~\ref{Int} illustrates an example of this graph (top left) together with the steps of the erasing procedure.

\begin{figure}[h!]
\centering
\tikzset{
  v/.style    = {circle, draw, fill=blue!60, inner sep=1.6pt},
  e/.style    = {line width=0.8pt},
  gone/.style = {line width=0.8pt, dashed},
  dead/.style = {circle, draw, fill=red!75, inner sep=1.8pt}
}

% ---------- параметры ----------
\def\GraphScale{0.99}
\def\XStep{4.0cm}   % горизонтальный шаг
\def\RowGap{7.0cm}  % вертикальный зазор
\def\s{5}           % s=5 => a0..a3, b0..b2

% ===== контроль рисования рёбер =====
\def\SkipList{} % сюда копим удалённые рёбра
\newcommand{\Edge}[1]{%
  \IfSubStr{\SkipList}{#1}{}{ \draw[e] #1; }%
}
\newcommand{\DashList}[1]{%
  \begingroup\edef\tmp{#1}\ifx\tmp\empty\else
    \foreach \E in {#1}{\draw[gone]\E;}
  \fi\endgroup
}

% ---------- исходный граф ----------
\newcommand{\DrawGraph}{%
  % удобные границы для последнего индекса
  \pgfmathtruncatemacro{\Amax}{\s-2} % последний индекс для a: s-2
  \pgfmathtruncatemacro{\Bmax}{\s-3} % последний индекс для b: s-3

  % путь P1: a0..a_{s-2} (в координатах используем i+1, чтобы сохранить прежние Y)
  \foreach \i [evaluate=\i as \yy using 2.6-0.6*(\i+1)] in {0,...,\Amax}{
    \node (A\i) at (0,\yy) [v] {};
    \node[left=2pt] at (A\i) {$a_{\ifnum\i=\Amax s-2\else\i\fi}$};
  }
  \foreach \i in {0,...,\numexpr\Amax-1\relax}{
    \pgfmathtruncatemacro{\ip}{\i+1}
    \Edge{(A\i)--(A\ip)}
  }

  % путь P2: b0..b_{s-3}
  \foreach \j [evaluate=\j as \yy using -1.2-0.6*(\j+1)] in {0,...,\Bmax}{
    \node (B\j) at (0,\yy) [v] {};
    \node[left=2pt] at (B\j) {$b_{\ifnum\j=\Bmax s-3\else\j\fi}$};
  }
  \foreach \j in {0,...,\numexpr\Bmax-1\relax}{
    \pgfmathtruncatemacro{\jp}{\j+1}
    \Edge{(B\j)--(B\jp)}
  }

  % C4-трапеция (основания ab и cd вертикальны и параллельны)
  \node (a) at (1.1,  1.00) [v] {};
  \node (b) at (1.1, -2.00) [v] {};
  \node (d) at (1.7,  0.00) [v] {};
  \node (c) at (1.7, -1.00) [v] {};
  \node[above right=1pt]  at (a) {$a$};
  \node[below right=1pt]  at (b) {$b$};
  \node[below right=1pt]  at (c) {$c$};
  \node[above right=1pt]  at (d) {$d$};

  % рёбра C4
  \Edge{(a)--(b)} \Edge{(c)--(d)} \Edge{(b)--(c)} \Edge{(a)--(d)}

  % соединения
  \Edge{(a)--(A0)} \Edge{(a)--(A1)} \Edge{(a)--(A2)} \Edge{(a)--(A3)}
  \Edge{(b)--(B0)} \Edge{(b)--(B1)} \Edge{(b)--(B2)}
}

\def\LabelX{0.0}\def\LabelY{-3.8}
\newcommand{\Under}[3]{%
  \begin{scope}[xshift=#1, yshift=#2]
    \node at (\LabelX,\LabelY) { #3};
  \end{scope}
}

% один кадр: (x, y, красная вершина, текущие удаляемые рёбра, ранее удалённые)
\newcommand{\Frame}[5]{%
  \begin{scope}[xshift=#1, yshift=#2]
    \begingroup
      \edef\SkipList{#5, #4}%
      \DrawGraph
      \DashList{#4}%
      \if\relax\detokenize{#3}\relax\else \node[dead] at (#3) {}; \fi
    \endgroup
  \end{scope}
}

% ===== шаги (с 0-нумерацией) =====
\def\DOne{(a)--(b)}                         \def\AccOne{\DOne}
\def\DTwo{(c)--(d)}                         \def\AccTwo{\AccOne, (c)--(d)}
\def\DThree{(b)--(c)}                       \def\AccThree{\AccTwo, (b)--(c)}
\def\DFour{(a)--(d)}                        \def\AccFour{\AccThree, (a)--(d)}
\def\DFive{(a)--(A0)}                       \def\AccFive{\AccFour, (a)--(A0)}
\def\DSix{(A0)--(A1)}                       \def\AccSix{\AccFive, (A0)--(A1)}
\def\DSeven{(a)--(A1),(a)--(A2),(a)--(A3),
            (A1)--(A2),(A2)--(A3)}
\edef\AccSeven{\AccSix, \DSeven}
\def\DEight{(b)--(B0),(b)--(B1),(b)--(B2),
            (B0)--(B1),(B1)--(B2)}
\edef\AccEight{\AccSeven,\DEight}

% ===== размещение кадров =====
\begin{tikzpicture}[scale=0.8, transform shape]
  \Frame{0*\XStep}{0cm}{}{}{}              \Under{0*\XStep}{0cm}{Original graph}
  \Frame{1*\XStep}{0cm}{d}{\DOne}{}        \Under{1*\XStep}{0cm}{Step 1}
  \Frame{2*\XStep}{0cm}{a}{\DTwo}{\AccOne} \Under{2*\XStep}{0cm}{Step 2}
  \Frame{3*\XStep}{0cm}{a}{\DThree}{\AccTwo}\Under{3*\XStep}{0cm}{Step 3}
  \Frame{4*\XStep}{0cm}{A0}{\DFour}{\AccThree}\Under{4*\XStep}{0cm}{Step 4}

  \Frame{1*\XStep}{-\RowGap}{A1}{\DFive}{\AccFour} \Under{1*\XStep}{-\RowGap}{Step 5}
  \Frame{2*\XStep}{-\RowGap}{a}{\DSix}{\AccFive}   \Under{2*\XStep}{-\RowGap}{Step 6}
  \Frame{3*\XStep}{-\RowGap}{}{}{\AccSeven}        \Under{3*\XStep}{-\RowGap}{Step 7}
  \Frame{4*\XStep}{-\RowGap}{}{}{\AccEight}        \Under{4*\XStep}{-\RowGap}{Step 8}
\end{tikzpicture}

\caption{Erasing process for the graph obtained from a disjoint $(s-1)$-path $P_1$, an $(s-2)$-path $P_2$, and a 4-cycle $(a,b,c,d)$, by joining $a$ to all vertices of $P_1$ and $b$ to all vertices of $P_2$. At each step, exactly one edge (dashed) is erased and exactly one vertex (red) is closed.}
\label{Int}
\end{figure}
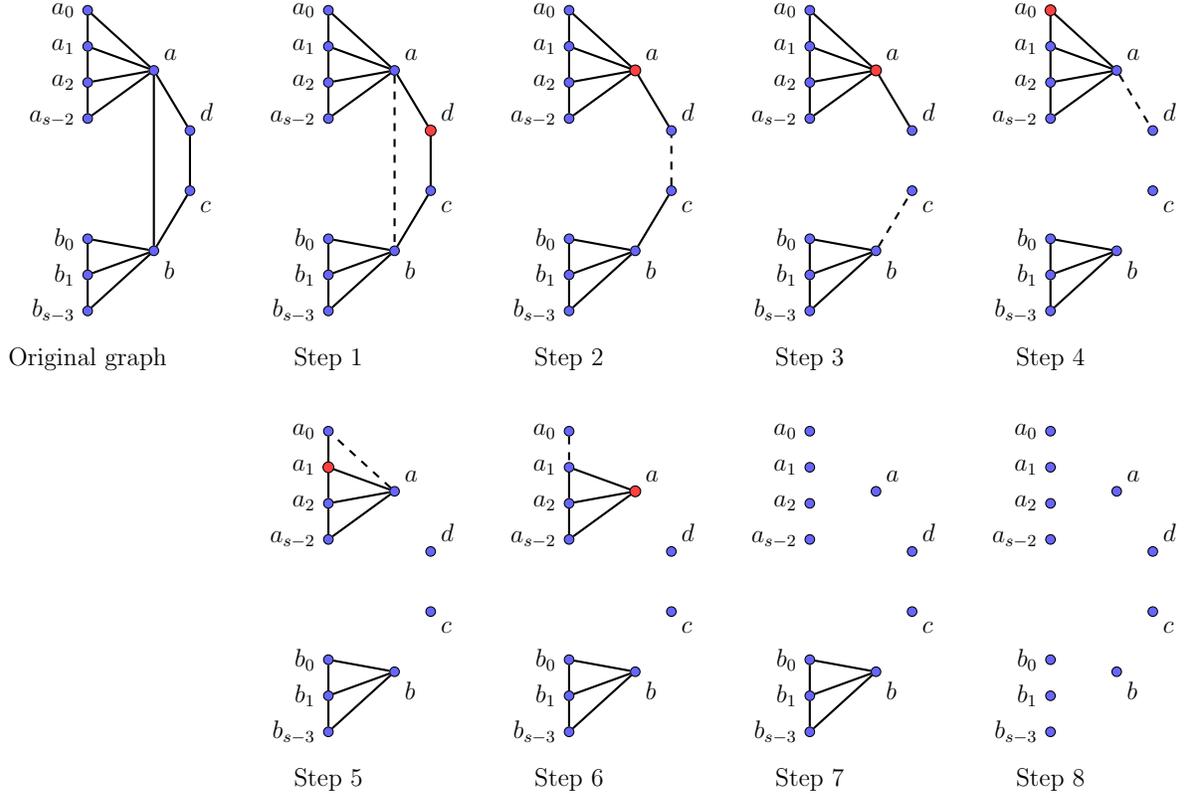

At each step, we colour the closed vertex in red. 
In Step~1, we close the vertex \(d\) and erase the edge $\{a,b\}$ --- the only edge between 
\(V^1_1 := \{a, a_0, \ldots, a_{s-2}\}\) and \(V^2_1 = \{c, b, b_0, \ldots, b_{s-3}\}\). 
Similarly, in Step~2, we close \(a\) and erase $\{c,d\}$, joining 
\(V^1_2 := \{d, a_0, \ldots, a_{s-2}\}\) with \(V^2_2 := \{c, b, b_0, \ldots, b_{s-3}\}\). 
In Step~3, we close \(a\) again and erase the edge $\{b,c\}$ between 
\(V^1_3 := \{c, a_0, \ldots, a_{s-2}\}\) and \(V^2_3 := \{d, b, b_0, \ldots, b_{s-3}\}\). 
In Step~4, we erase $\{a,d\}$ by closing, say, \(a_0\), where $\{a,d\}$ is the sole edge between 
\(V^1_4 := \{c, a, a_1, \ldots, a_{s-2}\}\) and \(V^2_4 := \{d, b, b_0, \ldots, b_{s-3}\}\). 
In Step~5, we close \(a_1\) and erase $\{a_0,a\}$. 
In Step~6, we close \(a\) and erase $\{a_0,a_1\}$. 
Similarly, we can erase the edges $\{a,a_i\}$ and $\{a_i,a_{i+1}\}$ for \(i = 1, \ldots, s-3\), the edge $\{a,a_{s-2}\}$, and likewise all the edges between \(b, b_0, \ldots, b_{s-3}\).

\subsection{Lower bound}

%\section{Proofs of Theorems \ref{result_1} and \ref{result_2}}

%In this section, we apply the hyperforest representation of the erase process and the semi-invariants \(s_i = f_i + 2c_i\) and \(Q\) introduced earlier. 

%\subsection{Lower Bound of Theorem \ref{result_1}}

We start the proof from the following auxiliary result that describes some scenarios that guarantee positiveness of $Q$.
%In Lemma \ref{auxiliary_lemma_for_Prop_2}, we isolate several recurring hyperforest configurations that guarantee a positive increase of \(Q\). These patterns will be used to secure the required total increment of $Q$ in the lower-bound proofs.

\begin{lemma}\label{auxiliary_lemma_for_Prop_2}
Let $G$ be an erasable graph and $H$ be its hyperforest. Then, there is an erase procedure satisfying the following. % \( G \) consist of distinct subgraphs as described below. Then, there exists a sequence of steps, each following Algorithm \ref{alg}, such that by the end of this sequence, the value of \( Q \) increases in each case as follows:
\begin{enumerate}
    \item If $H$ consists of an isolated vertex and two disjoint hyperedges of the same size bigger than 2, then $Q\geq 1$. \label{case1_of_auxiliary_lemma}
    \item If $H$ consists of two disjoint hyperedges $F_1,F_2$ such that $|F_2|=|F_1|+1\geq 4$, then $Q\geq 2$. \label{case2_of_auxiliary_lemma}
    \item If $H$ consists of an isolated hyperedge $F$ and hyperedges $F_1,F_2$ such that $|F_1\cap F_2|=1$ and $|F_1\cup F_2|=|F|+1\geq 4$, then $Q\geq 1$. \label{case3_of_auxiliary_lemma}
\end{enumerate}
\end{lemma}

\begin{proof}

We consider separately the three cases from the assertion of the lemma.

\begin{enumerate}
\item Let $H$ consist of an isolated vertex and two disjoint hyperedges of the same size bigger than 2.  As $G$ is erasable, it has an erasable edge $e$ and a closed vertex $v$. Note that $e$ and $v$ must belong to different hyperedges $F_e$ and $F_v$ since otherwise Erase$(e,G)$ is not possible. Therefore, Erase$(e,G)$ splits $F_e$ into two non-trivial sets $V_1$ and $V_2$, and one of these sets has size at least 2 as $|F_e|\geq 3$. Therefore, $(f_1,c_1)=(f_0,c_0)+(\lambda,1)$, where $\lambda\geq 0$. We get $Q_1\geq 1$ and then $Q\geq 1$ due to Corollary~\ref{S(H)}, as needed.

\item Similarly, if $E(H)=\{F_1,F_2\}$, where $F_1$ and $F_2$ are disjoint and $|F_2|=|F_1|+1\geq 4$, then an erasable edge $e$ and a closed vertex $v$ belong to different hyperedges. Again, Erase$(e,G)$ splits $F_e$ into $V_1,V_2$. Without loss of generality, $|V_2|\geq|V_1|$. If $|V_1|>1$, then a new edge is created implying $(f_1,c_1)\geq (f_0,c_0)+(1,1)$ that gives $Q_1\geq 2$ and $Q\geq 2$ due to Corollary~\ref{S(H)}, as needed. If $|V_1|=1$ and $|F_e|=|F_v|+1\geq 4$, then the edge-operation does not decrease the number of edges, and so $(f_1,c_1)= (f_0,c_0)+(\lambda,1)$, where $\lambda\geq 0$. If $\lambda\geq 1$, then $Q_1\geq 2$. If $\lambda=0$, then the only vertex-operation does not produce new edges meaning that the hyperforest $H_1$ at the next step of the process consists of one isolated vertex and two disjoint hyperedges of the same size $|F_1|\geq 3$ --- this reduces to the first case, where $Q_2\geq 1$. Therefore, $Q\geq 2$ due to Corollary~\ref{S(H)}. Finally, if $|V_1|=1$ and $|F_v|=|F_e|+1$, then $G[F_v\setminus\{v\}]$ must be disconnected, meaning that the vertex-operation applied to $F_v$ creates an additional hyperedge. Then, recalling that $|F_e|\geq 3$, we get $(f_1,c_1)\geq(f_0,c_0)+(1,1)$, implying $Q_1\geq 2$ and completing the proof.

\item If $E(H)=\{F,F_1,F_2\}$, where $|F_1\cup F_2|=|F|+1\geq 4$ and $|F_1\cap F_2|=1$, then an erasable edge $e$ and a closed vertex $v$ have to belong to different component $F_e$ and $F_v$, as well. If $|F_e|\geq 3$, then erasing $e$ does not decrease the number of edges, implying $Q_1\geq 1$ as $f_1\geq f_0$ and $c_1\geq c_0+1$. Without loss of generality, it remains to consider the case $e\in F_1$, $|F_1|=2$, and $|F_2|=|F|\geq 3$. If $f_1\geq f_0$, then $Q_1\geq 1$ as needed. If $f_1=f_0-1$, then the vertex operation applied to $F_v=F$ does not create new hyperedges, leading to the hyperforest $H_1$ consisting of one isolated vertex and two hyperedges of size $|F|\geq 3$ and reducing the problem to the first case. But then $Q_2\geq 1$, completing the proof due to Corollary~\ref{S(H)}.

\end{enumerate}

\end{proof}

We finish this section with the proof of the lower bound in Theorem \ref{result_1}, reiterated below.

\begin{proposition}
Let $G$ be an erasable graph on $2s+1$ vertices with $s>2$. Then $|E(G)|\leq 4s-4$.
\end{proposition}

\begin{proof}
Due to the bound~\eqref{eq:Q}, it suffices to show that there exists an erase process such that $Q\geq 3$ or, in other words, there exist steps $i\in I$ such that $Q_i\geq 1$ for $i\in I$ and $\sum_{i\in I}Q_i\geq 3$. Observe that if the graph \( G \) is erasable, then every spanning subgraph of \( G \) is also erasable. Consequently, it does not matter which erasable edge we select at each step --- we will eventually reach an empty graph always. We can assume that the initial graph \( G_0=G \) is either a connected graph with \((f_0, c_0) = (1,1)\) or a disjoint union of an isolated vertex and a connected graph on \( 2s \) vertices with \((f_0, c_0) = (1,2)\), since otherwise $f_0+2c_0 \geq 6$ implying $|E(G)|\leq 4s-4$ due to~\eqref{eq:Q}. 

If $(f_0,c_0)=(1,2)$, then $H_0$ is a hyperforest comprising an isolated vertex $u$ and one hyperedge $U$ of size $2s$. Due to the bound~\eqref{eq:Q} and Corollary~\ref{S(H)}, it suffices to show that either $Q_1\geq 1$, or $Q_2\geq 1$, or $\sum_{i\geq 4}Q_i\geq 1$. The only two transitions that give $Q_1=0$ are $(f_1,c_1)=(f_0,c_0)+(1,0)$ and $(f_1,c_1)=(f_0,c_0)+(-1,1)$. The second transition is impossible since it eliminates all hyperedges and then every vertex becomes a connected component meaning that $c_1=2s+1>2$. Then, we may assume that $(f_1,c_1)=(2,2)$, i.e. $U$ splits into two hyperedges $U_1,U_2$, where $|U_1|=s+1$, $|U_2|=s$, and $U_1\cap U_2=\{v_1\}$. As above, we may further assume that $(f_2,c_2)=(3,2)$, which is only possible if $U_1$ splits into two hyperedges $V$ and $\{v_1,v_2\}$ with one common vertex $v_2$. But then, at the third step we may close the isolated vertex $u=:v_3$ and erase the edge $\{v_1,v_2\}=:e_3$. Therefore, $H_3$ consists of the isolated vertex $u$ and two disjoint hyperedges $U_1\setminus\{v_1\}$ and $U_2$ of size $s$. By Lemma~\ref{auxiliary_lemma_for_Prop_2}, $\sum_{i\geq 4}Q_i\geq 1$, completing the proof.

It remains to consider the most complicated case $(f_0,c_0)=(1,1)$. By Lemma~\ref{events}, we have the following possibilities: $(f_1,c_1)=(2,1)$, $(f_1,c_1)=(2+\lambda,2)$, $(f_1,c_1)=(1,2)$, $(f_1,c_1)=(0,2)$, and $(f_1,c_1)=(3+\lambda,1)$, where $\lambda\geq 0$. The fourth case is actually impossible: if no hyperedges survived after the first erase procedure, then every vertex becomes a connected component, meaning that $c_1=2s+1>2$. The third case is also impossible: if $(f_1,c_1)=(1,2)$, then after the first erase procedure $H_1$ consists of an isolated vertex and a hyperedge $U$ of size $2s$. The isolated vertex may emerge only after the edge-operation. Since the vertex-operation is applied to $U$, the graph $G_0[U\setminus \{v_1\}]$ is connected, making the first erase procedure impossible. We also observe that, due to Corollary~\ref{S(H)}, $(f_1,c_1)\geq(3,2)$ immediately implies $Q\geq 3$. Therefore, it only remains to consider the following three cases: $(f_1,c_1)=(2,1)$, $(f_1,c_1)=(2,2)$, and $(f_1,c_1)=(3+\lambda,1)$.

\paragraph{Let $(f_1,c_1)=(2,1)$.} The hyperforest $H_1$ produced after the first Erase procedure has two hyperedges --- $F_1$ and $F_2$. Note $|F_1|=|F_2|=s+1$ and $F_1\cap F_2=\{v_1\}$. The next closed vertex $v_2$ should not be equal to $v_1$ since otherwise either $G_0[F_1\setminus\{v_1\}]$ or $G_0[F_2\setminus\{v_1\}]$ is disconnected meaning that the respective vertex-operation should have created more hyperedges at the first step. Moreover, the erasable edge $e_2$ and the vertex $v_2$ should lie in different hyperedges (call them $F_e$ and $F_v$, respectively). Indeed, if $F_v=F_e$, then erasing $e_2$ would not be possible, as the induced subgraphs $G[F_v]$ and $G[F_e]$ are both connected and have $s+1$ vertices. We get that $F_e=V_1\sqcup V_2$, where $v_1\in V_1$ and $e_2$ is the bridge between $V_1$ and $V_2$. We then consider several cases depending on whether $|V_1|=1$ or $|V_2|=1$. Since $|V_1|+|V_2|=s+1$, at least one of these sets has size greater than 1.

\begin{enumerate}

\item If $|V_1|>1$ and $|V_2|>1$, then the edge-operation applied at the second step splits $F_e$ into two disjoint hyperedges $V_1,V_2$ and creates an additional connected component. Moreover, this edge-operation is only possible if $G_1[F_v\setminus\{v_2\}]$ is disconnected. Therefore, the vertex-operation applied to $F_v$, replaces it with at least two hyperedges, see Figure~\ref{pic1}. Therefore, $Q_2=(f_2-f_1)+2(c_2-c_1)-1\geq 3$ implying $Q\geq 3$ due to Corollary~\ref{S(H)}, as needed. 
\begin{figure}[h]
\begin{tikzpicture}[scale=1.0, line width=0.9pt,baseline=(current bounding box.north)]
%\path[use as bounding box] (-3, -1.2) rectangle (10.8, 0.5);
\path[use as bounding box] (-3, -1.4) rectangle (14.2, 0.6);
\def\XR{1.3}
\def\YR{0.35}

%================ ЛЕВАЯ ЧАСТЬ =====================
\draw (-0.5,0) ellipse ({\XR} and {\YR});
\draw (1.6,0) ellipse ({\XR} and {\YR});

%\draw[line width=1.6pt] (-0.25,0.6) -- (-0.25,-0.6);
%\draw (-0.55,0.25) -- (-0.15,0.05);

% Вертикальное ребро внутри него (пунктир + жирные концы)
\draw[dashed, line width=1pt] (-0.5,0.25) -- (-0.5,-0.25);
\fill (-0.5,0.25) circle (2pt);
\fill (-0.5,-0.25) circle (2pt);
\node at (-0.2,0.05) {$e$};

\fill (0.6, 0) circle (2pt);
\node[above] at (0.6, 0.5) {$v_1$};

\node at (-0.5,-1.20) {$|F_1|=s+1$};
\node at (2.3,-1.20) {$|F_2|=s+1$};

\draw[->] (4.0,0) -- (5.0,0);

%================ ПРАВАЯ ЧАСТЬ =====================

% 1. Отдельная шапочка с плоским концом вправо
\begin{scope}[shift={(7.0,0)}]
  \draw[xscale=-1]
  (0,\YR) arc[start angle=90, end angle=-90,
        x radius=\XR, y radius=\YR] -- cycle;
\end{scope}

% 2. Шапочка с плоским концом влево
\begin{scope}[shift={(7.7,0)}]
  \draw[xscale=1]
    (0,\YR) arc[start angle=90, end angle=-90,
    x radius=\XR, y radius=\YR] -- cycle;
\end{scope}

% 3. Шапочка с плоским концом вправо
\begin{scope}[shift={(10.5,0)}]
  \draw[xscale=-1.5]
  (0,\YR) arc[start angle=90, end angle=-90,
        x radius=\XR, y radius=\YR] -- cycle;
\end{scope}

% 4. Эллипс
\draw (11.2,0) ellipse ({\XR} and {\YR});

% Вершина v
\fill (8.75, 0) circle (2pt);
\node[above] at (8.75, 0.5) {$v_1$};

% Вершина u
\fill (10.2,0) circle (2pt);
\node[above] at (10.2,0.5) {$v_2$};

\end{tikzpicture}

\caption{}
\label{pic1}
\end{figure}

\item If $|V_1|=s>2$ and $|V_2|=1$, then the edge-operation applied at the second step introduces a new connected component but does not increase the number of hyperedges. The erase procedure Erase$(e_2,G_1)$ with closed vertex $v_2$ separates the isolated vertex $V_2$ and the respective edge operation creates the hyperedge $V_1$ that sends edges only to $v_2$ in $F_v$. Moreover, similarly to the previous case, the vertex-operation contributes positively to the number of edges. If it splits $F_v$ in at least three hyperedges, then $Q_2\geq 3$, completing this case. Assume $F_v$ splits into two hyperedges. Since the edge-operation isolates the single vertex $V_2$, the vertex operation isolates a single vertex $u$ as well. Actually, $u=v_1$, since the common vertex $v_1$ of $V_1$ and $F_v$ do not have neighbours in $F_v\setminus\{v_1,v_2\}$, see Figure~\ref{pic2}. At this point, $Q_1+Q_2\geq 2$. %\textcolor{red}{need to mention that Q already got +2}%Then $v^1$ has neighbours in $G_1[G_v\setminus\{v^2\}]$. It means that `closing' the vertex $v^2$ and erasing $e^2$ leaves some neighbours of $G_1[V_1]$ in $F_v$ --- a contradiction with the definition of the Erase procedure. Therefore, $u=v^1$, see Figure~\ref{pic2}.

\begin{figure}[h]
\begin{tikzpicture}[scale=1.0, line width=0.9pt, baseline=(current bounding box.north)]
\path[use as bounding box] (-3, -1.4) rectangle (14.2, 0.6);

\def\XR{1.3}
\def\YR{0.35}

%================ ЛЕВАЯ ЧАСТЬ =====================
\draw (-0.5,0) ellipse ({\XR} and {\YR});
\draw (1.6,0) ellipse ({\XR} and {\YR});

%\draw[line width=1.6pt] (-0.25,0.6) -- (-0.25,-0.6);
%\draw (-0.55,0.25) -- (-0.15,0.05);
%\node at (0.05,0.05) {$e$};

% Вертикальное ребро внутри него (пунктир + жирные концы)
\draw[dashed, line width=1pt] (-0.5,0.25) -- (-0.5,-0.25);
\fill (-0.5,0.25) circle (2pt);
\fill (-0.5,-0.25) circle (2pt);
\node at (-0.2,0.05) {$e$};

\fill (-1.5,0) circle (2pt); % точка внутри левого эллипса
\node[above] at (-1.5,0.3) {$V_2$};
\fill (0.6,0) circle (2pt);  % v между левыми эллипсами
\node[above] at (0.6,0.3) {$v_1$};

\node at (-0.5,-1.0) {$|F_1|=s+1$};
\node at (2.3,-1.0) {$|F_2|=s+1$};

\draw[->] (4.0,0) -- (5.0,0);

%================ ПРАВАЯ ЧАСТЬ =====================

% левый большой эллипс и подпись s
\draw (7.6,0) ellipse ({\XR} and {\YR});
\node at (7.6,-1.0) {$s$};

% малый центральный эллипс (НЕ шапочка, НЕ линза)
\def\xr{0.72}
\def\yr{0.28}
\draw (9.2,0) ellipse ({\xr} and {\yr});
\draw[-] (8.65,0) -- (9.7,0);

% одиночная чёрная точка
\fill (5.5,0) circle (2.2pt);
\node[above] at (5.5,0.3) {$V_2$};

% точки v и u по краям малого эллипса
\fill (9.4-\xr,0) circle (2pt);
\node[above] at (9.4-\xr,0.3) {$v_1$};
\fill (9.05+\xr,0) circle (2pt);
\node[above] at (9.05+\xr,0.3) {$v_2$};
%\Edge{(v)--(u)}

% подпись 2 под малым эллипсом
\node at (9.2,-1.0) {$2$};

% правый большой эллипс и подпись s
\draw (10.9,0) ellipse ({\XR} and {\YR});
\node at (10.9,-1.0) {$s$};
\end{tikzpicture}
\caption{}
\label{pic2}
\end{figure}

Then, at the next step we may erase the edge $e_3=\{v_1,v_2\}$ by `closing' the isolated vertex $V_2$. The hyperforest $H_3$ is this isolated vertex with two disjoint hyperedges of size $s$. Due to Lemma \ref{auxiliary_lemma_for_Prop_2}, it contributes the missing 1 to $Q$.

\item If $|V_1|=1$ and $|V_2|=s>2$, then the edge-operation splits the hypergraph $H_1$ into two connected components --- $H_1[V_2]$ and $H_1[F_v]$. If the vertex-operation splits $F_v$ into at least three hyperedges, Corollary~\ref{S(H)} completes the proof. If $F_v$ splits into two hyperedges or does not change (see Figure~\ref{pic4}), the hyperforest $H_2$ contributes the missing 2 to $Q$ due to Lemma~\ref{auxiliary_lemma_for_Prop_2}, completing the proof.

\begin{figure}[h]
\begin{tikzpicture}[scale=1.0, line width=0.9pt,baseline=(current bounding box.north)]
% растягиваем bbox, чтобы центрировалось как нужно
\path[use as bounding box] (-3, -3.7) rectangle (13.2,1.3);

\def\XR{1.3}
\def\YR{0.35}

%================ ЛЕВАЯ ЧАСТЬ (исходная) =====================
\draw (-0.7,0) ellipse ({\XR} and {\YR});
\draw (1.4,0) ellipse ({\XR} and {\YR});

%\draw[line width=1.6pt] (-0.25,0.6) -- (-0.25,-0.6);
%\draw (-0.55,0.25) -- (-0.15,0.05);

\draw[dashed, line width=1pt] (-0.7,0.25) -- (-0.7,-0.25);
\fill (-0.7,0.25) circle (2pt);
\fill (-0.7,-0.25) circle (2pt);
\node at (-1.09,0.05) {$e$};

\fill (0.4, 0) circle (2pt);
\node[above] at (0.4,+0.5) {$v_1$};

\node at (-1.0,-1.20) {$|F_1|=s+1$};
\node at (2.0,-1.20) {$|F_2|=s+1$};

\draw[->] (3.0,0) -- (4.0,0);

%================ ПРАВАЯ ВЕРХНЯЯ ЛИНИЯ (λ = 1) =====================

% одиночный эллипс с подписью s
\draw (5.5,0) ellipse ({\XR} and {\YR});
\node at (5.5,-0.8) {$s$};

% пересечение: шапочка влево (как эллипс, с "плоским" стыком) + эллипс справа
% (делаем как два перекрывающихся эллипса для простоты визуально)
\draw (8.5,0) ellipse ({\XR} and {\YR});   % левый
\draw (10.4,0) ellipse ({\XR} and {\YR});  % правый

% точки v и u в зоне пересечения
\fill (7.8,-0.05) circle (2pt); \node[above] at (7.8,0.5) {$v_1$};
\fill (9.4,-0.05) circle (2pt); \node[above] at (9.4,0.5) {$v_2$};

\node at (12.6,-0.05) {$\lambda=1$};

%================ ПРАВАЯ НИЖНЯЯ ЛИНИЯ (λ = 0) =====================
% стрелка вниз-вправо
\draw[->] (3.0,-1.7) -- (4.0,-1.7);

% два раздельных эллипса с подписями s и s+1
\draw (5.5,-1.7) ellipse ({\XR} and {\YR});
\node at (5.5,-2.8) {$s$};

\draw (8.5,-1.7) ellipse ({\XR} and {\YR});
\node at (8.5,-2.8) {$s+1$};

\node at (12.6,-1.7) {$\lambda=0$};
\end{tikzpicture}
\caption{}
\label{pic4}
\end{figure}

\end{enumerate}

\paragraph{ Let $(f_1,c_1)=(2,2)$.} We get $Q_1=2$. Since $H_1$ consists of two disjoint hyperedges of sizes $s$ and $s+1$, it contributes an extra 2 to $Q$ due Lemma \ref{auxiliary_lemma_for_Prop_2}, completing the proof.

\paragraph{ Let $(f_1,c_1)=(3+\lambda,1)$, where $\lambda\geq 0$.} We have $Q_1=\lambda+1$. The hyperforest $H_1$ is connected and consists of $3+\lambda$ hyperedges that share one common vertex $v_1$. At every next step, whenever it is possible, we close $v_1$ and erase some edge. Each such move increases the number of hyperedges by 1 and keeps the hypergraph connected. Assume that $t$ is the last such step. We get a hypertree $H_t$ with $t+2+\lambda$ hyperedges $F_1,F_2,\ldots, F_{t+2+\lambda}$ such that any two hyperedges have the only common vertex $v_1$. Note that every hyperedge has size at most $s+1$. The graph $G_t$ is still non-empty and erasable.

It remains to prove that $\sum_{i\geq t+1}Q_i\geq 2$. We then consider the step $t+1$, where we close a vertex $v_{t+1}$ and erase an edge $e_{t+1}$ in $G_{t+1}$. Observe that the vertex $v_{t+1}$ belongs to exactly one $F_v:=F_j$, $j\in[t+2+\lambda]$, and $e\notin F_v$ as the graph $G_{t+1}[V(G_{t+1})\setminus(F_v\setminus\{v_{t+1}\})]$ is connected and has size at least $s+1$. Denote the hyperedge containing $e$  by $F_e$. Let $F_e=V_1\sqcup V_2$,  where $v_1\in V_1$ and $e$ is bridge between $V_1$ and $V_2$. The edge-operation disconnects $V_2$ and, therefore, creates a new connected component. Consider the following cases:

\begin{enumerate}

\item If $|V_1|>1$, we get $Q_{t+1}\geq 2$, unless $|V_2|=1$. In the latter case, the vertex-operation must disconnect $v_1$ from $F_v\setminus\{v_1,v_{t+1}\}$, creating an additional hyperedge and giving $Q_{t+1}\geq 2$, as needed. 

\item If $|V_1|=|V_2|=1$, then the vertex-operation disconnects $v_1$ from $F_v\setminus\{v_1,v_{t+1}\}$ that has to have size $s-1$ (joining this set with $V_2$ makes up one of the two parts of the partition witnessing the erase procedure for $e_{t+1}$). Recalling that, during the first $t$ steps, we erased bridges in $G_0\setminus v_1$ between sets of size $s$, one of these sets was always $F_v\setminus v_1$. But then there could be only one bridge, i.e. $t=1$. At step $t+1=2$, the total number of hyperedges does not decrease and the total number of connected components increases, i.e. $Q_2\geq 1$. If $\lambda\geq 1$ or $Q_2\geq 2$, then the proof is already completed. So, it only remains to consider the case $\lambda=0$ and $f_2=f_1$. Since $\lambda=0$ and $t=1$, we have $t+2+\lambda=3$, i.e. after the first step we had three hyperedges --- two hyperedges of size $s+1$ and one hyperedge $F_e$ of size 2. After the second step, we have the isolated vertex $V_2$ and three hyperedges $F'_1,F'_2,F'_3$, where $F_v=F'_1\cup F'_2$ and $F'_3$ share $v_1$, $F'_1,F'_2$ share $v_{t+1}$, and $F'_2$ consists of the single edge $\{v_1,v_{t+1}\}$. At the third step, let us close the isolated vertex $V_2$ and erase the edge $\{v_1,v_2\}$. We get that $H_3$ comprises the isolated vertex and two isolated hyperedges of size $s$.  Due to Lemma \ref{auxiliary_lemma_for_Prop_2}, $\sum_{i\geq 4}Q_i\geq 1$, completing the proof.

\item If $|V_1|=1$ and $2\leq |V_2|\leq s-1$, then the edge-operation increases the number of connected components and does not change the number of hyperedges, while the vertex-operation increases the number of hyperedges, implying $Q_{t+1}\geq 2$.

\item If $|V_1|=1$ and $|V_2|=s$, then the edge-operation creates an isolated hyperedge of size $s$ and does not change the number of hyperedges, i.e. $Q_{t+1}\geq 1$. As in the second case ($|V_1|=|V_2|=1$), we may assume that $t=1$, $\lambda=0$, and $f_2=f_1$. Therefore, $H_2$ comprises an isolated hyperedge of size $s$ and two other hyperedges that share one common vertex. Due to Lemma \ref{auxiliary_lemma_for_Prop_2}, $\sum_{i\geq 3}Q_i\geq 1$, completing the proof.

\end{enumerate}

\end{proof}

\section{Proof of Theorem~\ref{result_2}}
\label{sc:2_proof}

\subsection{Upper bound}

We prove the upper bound in Theorem~\ref{result_2} separately in the two cases $gcd(s,t)=1$ and $gcd(s,t)\neq 1$.

\subsubsection{\( \gcd(s, t) = 1 \)}\label{Exm1}

The example of an erasable graph with \(2s + 2t - 2\) edges and \(s + t + 1\) vertices is presented in Figure~\ref{im2}. 
This graph consists of two disjoint paths \(P_a = (a_0, \ldots, a_{t-1})\) and \(P_b = (b_0, \ldots, b_{s-2})\) and two adjacent vertices \(a, b\), 
where \(a\) is adjacent to every vertex of \(P_a\), \(b\) is adjacent to every vertex of \(P_b\) and to \(a_{t-1}\).

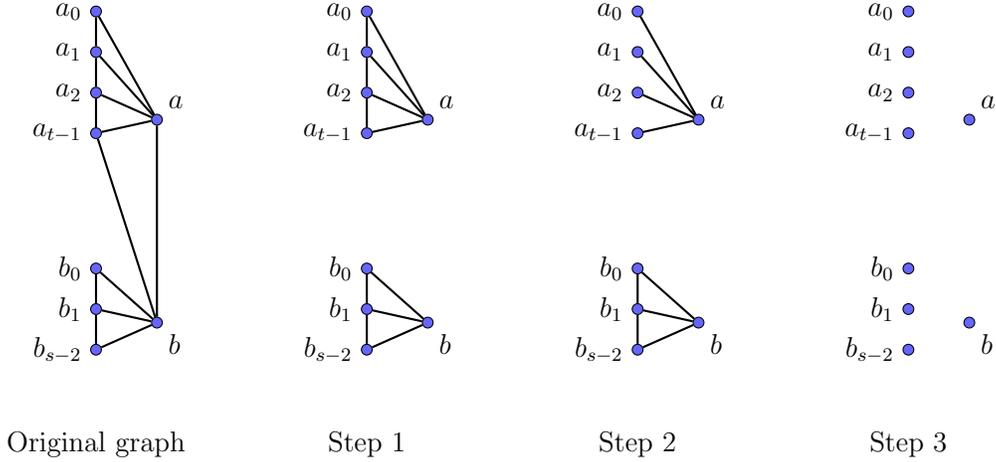
\begin{figure}[h!]
\centering
\tikzset{
  v/.style    = {circle, draw, fill=blue!60, inner sep=1.6pt},
  e/.style    = {line width=0.8pt},
  gone/.style = {line width=0.8pt, dashed},
  dead/.style = {circle, draw, fill=red!75, inner sep=1.8pt}
}

% ---------- параметры ----------
\def\GraphScale{0.99}
\def\XStep{4.0cm}   % горизонтальный шаг
%\def\RowGap{7.0cm}  % вертикальный зазор
%\def\s{5}           % s=5 => a1..a4, b1..b3

% ===== контроль рисования рёбер =====
\def\SkipList{} % сюда копим удалённые рёбра
\newcommand{\Edge}[1]{%
  \IfSubStr{\SkipList}{#1}{}{ \draw[e] #1; }%
}
\newcommand{\DashList}[1]{%
  \begingroup\edef\tmp{#1}\ifx\tmp\empty\else
    \foreach \E in {#1}{\draw[gone]\E;}
  \fi\endgroup
}

% ---------- исходный граф ----------
\newcommand{\DrawGraph}{%
  % путь P1: a1..a4
  \foreach \i [evaluate=\i as \yy using 2.6-0.6*\i] in {0,1,2,3}{
    \node (A\i) at (-2,\yy) [v] {};
    \node[left=2pt] at (A\i) { $a_{\ifnum\i=3 t-1\else\i\fi}$};
  }
  \Edge{(A0)--(A1)} \Edge{(A1)--(A2)} \Edge{(A2)--(A3)}

  % путь P2: b1..b3
  \foreach \j [evaluate=\j as \yy using -1.2-0.6*\j] in {0,1,2}{
    \node (B\j) at (-2,\yy) [v] {};
    \node[left=2pt] at (B\j) { $b_{\ifnum\j=2 s-2\else\j\fi}$};
  }
  \Edge{(B0)--(B1)} \Edge{(B1)--(B2)}

  % C4-трапеция (основания ab и cd вертикальны и параллельны)
  \node (a) at (-1.1,  1.00) [v] {};
  \node (b) at (-1.1, -2.00) [v] {};
  \node[above right=1pt]  at (a) {$a$};
  \node[below right=1pt]  at (b) {$b$};

  % рёбра C4
  \Edge{(a)--(b)} \Edge{(A3)--(b)}

  % соединения
  \Edge{(a)--(A0)} \Edge{(a)--(A1)} \Edge{(a)--(A2)} \Edge{(a)--(A3)}
  \Edge{(b)--(B0)} \Edge{(b)--(B1)} \Edge{(b)--(B2)}
}
\def\LabelX{1.65}\def\LabelY{-4.8}
% один кадр: (x, y, красная вершина, текущие удаляемые рёбра, ранее удалённые)
\newcommand{\Frame}[5]{%
  \begin{scope}[xshift=#1, yshift=#2]
    \begingroup
      % список рёбер, которые нужно пропустить (уже удалены или удаляются сейчас)
      \edef\SkipList{#5, #4}%
      % рисуем граф без этих рёбер
      \DrawGraph
      % пунктиром — только удаляемые на этом шаге рёбра
      \DashList{#4}%
      % проверка на пустой аргумент с красной вершиной
      \if\relax\detokenize{#3}\relax
        % пусто — пропускаем
      \else
        \node[dead] at (#3) {};
      \fi
    \endgroup
  \end{scope}
}

\def\LabelX{-2.0}\def\LabelY{-3.8}
\newcommand{\Under}[3]{%
  \begin{scope}[xshift=#1, yshift=#2]
    \node at (\LabelX,\LabelY) { #3};
  \end{scope}
}
% ---- сюда вставляешь весь твой tikzpicture ----
%\begin{tikzpicture}[scale=\GraphScale, transform shape]
\begin{tikzpicture}[scale=0.9, transform shape]

% ===== шаги (по твоему тексту) =====
% 1) close d, erase ab
\def\DOne{(a)--(b)}                         \def\AccOne{(a)--(b),(A3)--(b)}
% 2) close a, erase cd
\def\DTwo{}                         
\def\AccTwo{(a)--(b), (A3)--(b),(A0)--(A1),(A1)--(A2),(A2)--(A3)}
% 3) close a, erase bc
\def\DThree{}                       
\def\AccThree{\AccTwo, (B0)--(B1),(B1)--(B2),(b)--(B0),(b)--(B1),(b)--(B2),(a)--(A0),(a)--(A1),(a)--(A2),(a)--(A3)}
% 4) close a1, erase ad

% ===== ряд 1: исходник + шаги 1–4 =====

\Frame{0*\XStep}{0cm}{}{}{}              % исходный граф
\Under{0*\XStep}{0cm}{Original graph}   
%\Frame{1*\XStep}{0cm}{}{\DOne}{}   % шаг 1
%\Under{1*\XStep}{0cm}{Step 1}
\Frame{1*\XStep}{0cm}{}{\DTwo}{\AccOne}    % шаг 2
\Under{1*\XStep}{0cm}{Step 1}
\Frame{2*\XStep}{0cm}{}{\DThree}{\AccTwo}  % шаг 3
\Under{2*\XStep}{0cm}{Step 2}
\Frame{3*\XStep}{0cm}{}{}{\AccThree}  % шаг 3
\Under{3*\XStep}{0cm}{Step 3}
\end{tikzpicture}
\caption{The example of an erasable graph with \(2s + 2t - 2\) edges and \(s + t + 1\) vertices and its erasing process for the case \(\gcd(s, t) = 1\). }
\label{im2}
\end{figure}

%\begin{figure}[h]
%\includegraphics[width=0.8 \linewidth]{Picture_14.jpg}
%\caption{A graph with \( 2s + 2t - 2 \) edges when \( \gcd(s, t) = 1 \).}
%\label{im2}
%\end{figure}

Figure~\ref{im2} also illustrates the algorithm for erasing all edges in the graph. In Step~1, we  
1) close the vertex \(a\) and delete the edge \(\{a_{t-1}, b\}\), and  
2) close the vertex \(a_0\) and delete the edge \(\{a, b\}\). 
In Step~2, we erase all edges of the path \(P_a\) as follows: first, we close \(a\) and erase the edge \(\{a_{s-1}, a_s\}\), and then, in a similar way, we erase one by one the edges \(\{a_{h\cdot s - 1}, a_{h\cdot s}\}\) for \(h \geq 2\), where the product and the sum operations are taken in \(\mathbb{Z}_t\) considered as a left \(\mathbb{Z}\)-module. Since \(\gcd(s,t) = 1\), we eventually erase all edges of \(P_a\) in Step~2.  
In Step~3, we close \(b\) and erase, one by one, the edges \(\{a_0, a\}, \{a_1, a\}, \ldots, \{a_{t-1}, a\}\).  
In Step~4, for every remaining edge \(\{u, v\}\) with \(u, v \in \{b, b_0, \ldots, b_{s-2}\}\) and \(u\) having degree at most \(2\) in the remaining graph, we close either \(a\) or the other neighbour of \(u\) (if it exists) and erase this edge via the partition with one part \(\{a_0, \ldots, a_{t-2}, u\}\), say.  
Thus, all edges are eventually erased.

\subsubsection{\( \gcd(s, t) > 1 \)}

The example of an erasable graph with \(2s + 2t - 3\) edges and \(s + t + 1\) vertices is presented in Figure~\ref{im3}. 
It is obtained from the graph in Figure~\ref{im2} by deleting the edge \(\{a, a_{t-1}\}\).

\begin{figure}[h!]
\centering
\tikzset{
  v/.style    = {circle, draw, fill=blue!60, inner sep=1.6pt},
  e/.style    = {line width=0.8pt},
  gone/.style = {line width=0.8pt, dashed},
  dead/.style = {circle, draw, fill=red!75, inner sep=1.8pt}
}

% ---------- параметры ----------
\def\GraphScale{0.99}
\def\XStep{3.6cm}   % горизонтальный шаг
%\def\RowGap{7.0cm}  % вертикальный зазор
%\def\s{5}           % s=5 => a1..a4, b1..b3

% ===== контроль рисования рёбер =====
\def\SkipList{} % сюда копим удалённые рёбра
\newcommand{\Edge}[1]{%
  \IfSubStr{\SkipList}{#1}{}{ \draw[e] #1; }%
}
\newcommand{\DashList}[1]{%
  \begingroup\edef\tmp{#1}\ifx\tmp\empty\else
    \foreach \E in {#1}{\draw[gone]\E;}
  \fi\endgroup
}

% ---------- исходный граф ----------
\newcommand{\DrawGraph}{%
  % путь P1: a1..a4
  \foreach \i [evaluate=\i as \yy using 2.6-0.5*\i] in {0,1,2,3,4,5,6,7,8}{
    \node (A\i) at (-2,\yy) [v] {};
    \node[left=2pt] at (A\i) { $a_{\ifnum\i=8 t-1\else\i\fi}$};
  }
  \Edge{(A0)--(A1)} \Edge{(A1)--(A2)} \Edge{(A2)--(A3)}
  \Edge{(A3)--(A4)} \Edge{(A4)--(A5)} \Edge{(A5)--(A6)}\Edge{(A6)--(A7)} \Edge{(A7)--(A8)}

  % путь P2: b1..b3
  \foreach \j [evaluate=\j as \yy using -2.0-0.5*\j] in {0,1}{
    \node (B\j) at (-2,\yy) [v] {};
    \node[left=2pt] at (B\j) { $b_{\ifnum\j=1 s-2\else\j\fi}$};
  }
  \Edge{(B0)--(B1)} 

  % C4-трапеция (основания ab и cd вертикальны и параллельны)
  \node (a) at (-0.7,  0.60) [v] {};
  \node (b) at (-0.7, -2.70) [v] {};
  \node[above right=1pt]  at (a) {$a$};
  \node[below right=1pt]  at (b) {$b$};

  \Edge{(a)--(b)} \Edge{(b)--(A8)} 

  % соединения
  \Edge{(a)--(A0)} \Edge{(a)--(A1)} \Edge{(a)--(A2)} \Edge{(a)--(A3)} \Edge{(a)--(A4)} \Edge{(a)--(A5)} \Edge{(a)--(A6)} \Edge{(a)--(A7)} 
  \Edge{(b)--(B0)} \Edge{(b)--(B1)} 
}
\def\LabelX{1.65}\def\LabelY{-3.8}
% один кадр: (x, y, красная вершина, текущие удаляемые рёбра, ранее удалённые)
\newcommand{\Frame}[5]{%
  \begin{scope}[xshift=#1, yshift=#2]
    \begingroup
      % список рёбер, которые нужно пропустить (уже удалены или удаляются сейчас)
      \edef\SkipList{#5, #4}%
      % рисуем граф без этих рёбер
      \DrawGraph
      % пунктиром — только удаляемые на этом шаге рёбра
      \DashList{#4}%
      % проверка на пустой аргумент с красной вершиной
      \if\relax\detokenize{#3}\relax
        % пусто — пропускаем
      \else
        \node[dead] at (#3) {};
      \fi
    \endgroup
  \end{scope}
}

\def\LabelX{-1.2}\def\LabelY{-3.8}
\newcommand{\Under}[3]{%
  \begin{scope}[xshift=#1, yshift=#2]
    \node at (\LabelX,\LabelY) { #3};
  \end{scope}
}
% ---- сюда вставляешь весь твой tikzpicture ----
%\begin{tikzpicture}[scale=\GraphScale, transform shape]
\begin{tikzpicture}[scale=0.9, transform shape]

% ===== шаги (по твоему тексту) =====
% 1) close d, erase ab
\def\DOne{(a)--(b)}                         \def\AccOne{(a)--(b),(b)--(A8)}
% 2) close a, erase cd
\def\DTwo{}                         
\def\AccTwo{(a)--(b), (A5)--(A6),(A2)--(A3),(b)--(A8)}
% 3) close a, erase bc
\def\AccThree{(a)--(b), (A5)--(A6),(A2)--(A3),(b)--(A8),(A7)--(A8)}                      
\def\AccFour{\AccTwo, (B0)--(B1),(B1)--(B2),(b)--(B0),(b)--(B1),(b)--(B2),(a)--(A0),(a)--(A1),(a)--(A2),(a)--(A3), (a)--(A4),(a)--(A5),(a)--(A6),(a)--(A7), (A0)--(A1),(A1)--(A2),(A3)--(A4), (A6)--(A7),(A7)--(A8), (A4)--(A5)}
% 4) close a1, erase ad

% ===== ряд 1: исходник + шаги 1–4 =====

\Frame{0*\XStep}{0cm}{}{}{}              % исходный граф
\Under{0*\XStep}{0cm}{Original graph}   
%\Frame{1*\XStep}{0cm}{}{\DOne}{}   % шаг 1
%\Under{1*\XStep}{0cm}{Step 1}
\Frame{1*\XStep}{0cm}{}{\DTwo}{\AccOne}    % шаг 2
\Under{1*\XStep}{0cm}{Step 1}
\Frame{2*\XStep}{0cm}{}{}{\AccTwo}  % шаг 3
\Under{2*\XStep}{0cm}{Step 2}
\Frame{3*\XStep}{0cm}{}{}{\AccThree}  % шаг 3
\Under{3*\XStep}{0cm}{Step 3}
\Frame{4*\XStep}{0cm}{}{}{\AccFour}  % шаг 3
\Under{4*\XStep}{0cm}{Steps 4-6}
\end{tikzpicture}
\caption{The example of an erasable graph with \(2s + 2t - 3\) edges and \(s + t + 1\) vertices and its erasing process for the case \(\gcd(s, t) \neq 1\).}
\label{im3}
\end{figure}
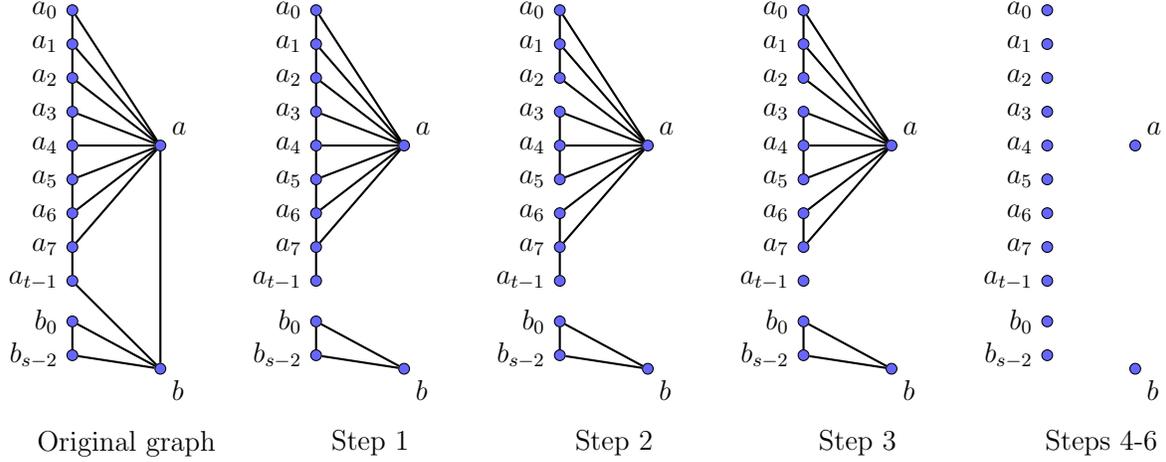

%\begin{figure}[h]
%\includegraphics[width=1.0 \linewidth]{Picture_15.jpg}
%\caption{A graph with \(2s+2t-3\) edges when \(gcd(s,t)\neq 1\).}
%\label{im3}
%\end{figure}

Let us show that this graph is erasable, see Figure~\ref{im3}.
In the same way as in the previous case, in Step~1, we erase the edges $\{a_{t-1},b\}$ and $\{a,b\}$.
Then, similarly, in Step~2, we erase all edges $\{a_{h\cdot s-1},a_{h\cdot s}\}$, $h\geq 1$.
It leaves disjoint $\gcd(s,t)$-paths.
In Step~3, we close $b$ and delete the edge $\{a_{t-2},a_{t-1}\}$ --- the only edge between $\{a,a_0,\ldots,a_{t-2}\}$ and $\{a_{t-1},b_0,\ldots,b_{s-2}\}$.
In Step~4, for every $j=0,1,\ldots,t-2$, if the edge $\{a_j,a_{j+1}\}$ has not been erased, we close $a$ and erase this edge --- the only edge between $\{a_{j+1},\ldots,a_{j+(s-1)-r},a_{t-1},\ldots,a_{t-1+r}\}$ and the remaining part of the graph, where the indices are taken from $\mathbb{Z}_t$ and $r$ is the remainder after dividing $j$ by $\gcd(s,t)$.
In Step~5, we delete edges $\{a_0,a\},\ldots,\{a_{t-2},a\}$ by closing $b$.
Step~6 --- removing all edges in the $b$-part of the graph --- is analogous to the respective step in the case $\gcd(s,t)=1$.

\subsection{Lower bound}

Let \( G \) be an erasable graph on \( s + t + 1 \) vertices. Due to Claim~\ref{claim}, we need to show that $|E(G)|\leq 2s+2t-2$. Consider an erase process of $G$. Due to Corollary~\ref{S(H)}, it suffices to  show that $Q_i\geq 1$ for some $i\geq 1$, assuming $(f_0, c_0) = (1,1)$. Also, due to the discussion after this corollary, we may assume $(f_1, c_1) = (2,1)$, since otherwise $Q_1\geq 1$. In other words, the hyperforest \( H_1 \) is connected and consists of two hyperedges of sizes $s+1$ and $t+1$, and sharing a vertex \( v_1 \). 

We may assume that every next step $i$ has either type $(f_{i-1},c_{i-1})\to (f_{i-1}+1,c_{i-1})$ or type $(f_{i-1},c_{i-1})\to (f_{i-1}-1,c_{i-1}+1)$.
%We then perform the following maximal sequence of steps: at every step $j\geq 2$, close $v_1$ and erase an edge $e_j$, if such an edge exists. We may assume that every step $j$ in this sequence has either type $(f_{j-1},c_{j-1})\to (f_{j-1}+1),c_{j-1})$ or $(f_{j-1},c_{j-1})\to (f_{j-1}-1,c_{j-1}+1)$.
 The latter type is only possible when a hyperedge of size 2 got erased. Assume $j$ is the first step of such type. In particular, $c_{j-1}=1$. 

\begin{claim}
The hypergraph $H_{j-1}$ has one of the following two shapes: either 
\begin{enumerate}
\item it is a hyperstar, i.e. all hyperedges share the same common vertex, or
\item the unique hyperedge of size 2 joins two centers of two hyperstars, and all the other hyperedges have size $s+1$.
\end{enumerate}
\end{claim}
 
\begin{proof}
We prove that actually the same holds for all hypergraphs $H_i$, $i\leq j-1$, by induction over $i$, where in the second case the hyperedge joining two centers has arbitrary size. Assume $i\leq j-2$, and $H_i$ has the first shape, i.e. it is a hyperstar with center $v_1$. If $v_{i+1}=v_1$, then it remains a hyperstar. Otherwise, $v_{i+1}$ and $e_{i+1}$ belong to the same hyperedge $F$, and by the definition of the erase procedure, the hyperedge $F$ splits into two hyperedges, where one of the hyperedges has size $s+1$. Moreover, if this hyperedge contains $v_1$, then the second hyperedge also has size $s+1$. It means that $H_{i+1}$ has the second shape.

Now, let us assume that $H_i$ has the second shape. Then, by the definition of the erase procedure and since $Q_{i+1}=0$, the edge $F$ joining two centers have size more than $s+1$, the edge $e_{i+1}$ and the vertex $v_{i+1}$ belong to $F$. The edge-operation and the vertex-operation split $F$ into two hyperedges $V_1,V_2$, where $V_1$ contains both centers and $V_2$ has size $s+1$. If $v_{i+1}$ equals one of the two centers, we are done. Otherwise, we get a hypergraph of a different shape, where all hyperedges, other than $F$, have one common vertex with $F$, and the number of vertices that $F$ shares with other hyperedges is more than 2. It is easy to see by induction that all subsequent erase procedures preserve this structure. In particular, eventually we get that either the ``central'' hyperedge has size 2 or a ``leaf'' hyperedge has size 2. The first option is impossible since the central hyperedge has to contain at least 3 vertices. The last option is also impossible since the $(j-1)$-th erase procedure should separate a set of size exactly $t$ by ``closing'' the vertex $v_{j-1}$, thus the $(j-1)$-th ``hyperleaf'' should have size exactly $t+1>2$.
\end{proof}

Assume $H_{j-1}$ has the second shape. We may assume that $Q_j=0$. Then $H_j$ consists of two connected components --- an isolated hyperedge of size $s+1$ and a hyperstar of size $t$. It is easy to see that $e_{j+1}$ and $v_{j+1}$ belong to different hyperedges, giving $Q_{j+1}\geq 1$.

%Moreover, this new hyperedge has a unique vertex that belongs to some other hyperedges of the hypertree, implying that each hyperedge $F_2,\ldots,F_{i+1}$ of size $t+1$ in $H_i$ has this property. Therefore, $H_{j-1}$ is a hypertree with all but one hyperedges of size $t+1$ (with this property) and one hyperedge of size $2$. We call such hypertree {\it amenable}. Since $G$ is erasable, an amenable hypertree is always reachable.

%Let us note that the amenable hypertree $H_{j-1}$ has one of the following two shapes: either it is a hyperstar (i.e. all hyperedges share the same common vertex), or the hyperedge $\{u,v\}$ of size 2 joins two centers of two hyperstars. Assume $H_{j-1}$ has the second shape. First of all, it means that all 

%Assume the next erased edge is not $\{u,v\}$. If this edge is a bridge of some hyperedge, then $Q_j\geq 1$. Therefore, the edge-operation and the vertex-operation are applied to the same hyperedge. It is easy to see that in this case the erase procedure is only possible when $H_j$ has the second type --- union of two hyperstars.

%If $H_{j-1}$ has the second type, the edge-operation is applied to $\{u,v\}$, and $Q_j=0$, then $H_j$ consists of two connected components --- an isolated hyperedge of size $t+1$ and a hyperstar. It is easy to see that $e_{t+1}$ and $v_{t+1}$ belong to different hyperedges, giving $Q_{j+1}\geq 1$.

It remains to consider the case when $H_{j-1}$ is a hyperstar. The closed vertex $v_j$ does not belong to $e_j$, meaning that it is into the `interior' of some other hyperedge $F$. Then, closing this vertex and erasing $e_j$ leaves a connected component of size at least $t$. Since its size should be equal to $t$, the size of $F$ is $s+1$, and closing $v_j$ splits it into at least two hyperedges, that is $f_j\geq f_{j-1}$ and $c_j=c_{j-1}+1$. Therefore, $Q_j\geq 1$, completing the proof.

\section{Proof of Theorem~\ref{result_3}}
\label{sc:3_proof}

\subsection{Upper bound}

Here we construct a weakly $K_{s,t}$-saturated graph $G$ on $s+t+j$ vertices with 
$$
 {s+t+j\choose 2}-j(s+t-2)-(2t-3)=(st-1) + \binom{s}{2} + \binom{t-2-j}{2} + j(t-1) + 1
$$
edges, which is sufficient for the upper bound on $\mathrm{wsat}(s+t+j,K_{s,t})$ from Theorem~\ref{result_3}. The graph $G$ is presented on Figure \ref{example.jpg}. It consists of three sets of vertices $A=\{a_1,\ldots,a_s\}$, $B=\{b_1,\ldots,b_t\}$, $C=\{c_1,\ldots,c_j\}$ and the following edges:
\begin{itemize}
\item all edges between vertices in $A$, i.e. $G[A]\cong K_s$,
\item all edges between $b_i,b_{i'}$, where $j+3\leq i<i'\leq t$,
\item all edges between $A$ and $B$ except for the edge $\{a_1,b_1\}$,
\item all edges between $B\setminus\{b_1\}$ and $C$,
\item the edge $\{b_1,b_{j+3}\}$.
\end{itemize}

\begin{figure}[h]
\centering
\begin{tikzpicture}[line join=round, line cap=round, scale=1, line width=0.8pt]

  % ===== ПАРАМЕТРЫ =====
  \newcommand{\NA}{6}   % |A| = s
  \newcommand{\JC}{3}   % |C| = j
  \pgfmathtruncatemacro{\JJT}{\JC+3}   % индекс j+3
  \pgfmathtruncatemacro{\TB}{\JJT+2}   % t = j+5
  \pgfmathtruncatemacro{\Bfront}{\JC+2} % b_1..b_{j+2} на овале

  \def\a{1.6}   % полуось эллипсов по x
  \def\b{0.48}  % полуось эллипсов по y
  \def\h{2.6}   % высота
  \pgfmathsetmacro{\dx}{2*\a + 0.9} % сдвиг C вправо

  % Внутренний «круг» (эллипс) в основании B
  \def\rinscale{0.50}
  \pgfmathsetmacro{\rinA}{\rinscale*\a}
  \pgfmathsetmacro{\rinB}{\rinscale*\b}

  % Углы (вершины не рисуем)
  \def\phiA{210}
  \def\phiB{270}
  \def\phiC{330}

  % Уводим a1, b1 внутрь + вверх
  \def\inset{0.90}
  \def\dyA{+0.15}
  \def\dyB{+0.12}

  % Точки касания наклонного цилиндра внутри эллипсов
  \def\phiTop{18}   % верхний эллипс C
  \def\phiBase{12}  % основание B

  % Реже пунктир для всех dashed-линий
  \tikzset{
    dashloose/.style={dash pattern=on 4.5pt off 3.5pt},
    missing/.style={dashloose},  % отсутствующие рёбра
    hidden/.style={dashloose},   % задняя образующая наклонного цилиндра
  }

  % ===== Центры =====
  \coordinate (O) at (0,0);        % B
  \coordinate (T) at (0,\h);       % A
  \coordinate (U) at (\dx,\h);     % C

  % ===== ШТРИХОВКИ (сначала) =====
  \path[pattern=north east lines, pattern color=black]
    (T) ellipse [x radius=\a, y radius=\b];            % основание A
%  \path[pattern=north east lines, pattern color=black]
%    (U) ellipse [x radius=\a, y radius=\b];            % верхнее основание C
  \path[pattern=north east lines, pattern color=black]
    (O) ellipse [x radius=\rinA, y radius=\rinB];       % малый круг в B

  % ===== Контуры основания B и внутреннего эллипса =====
  \draw (O) ++(0:\a)
    arc[start angle=0,  end angle=180, x radius=\a, y radius=\b];
  \draw (O) ++(180:\a)
    arc[start angle=180,end angle=360, x radius=\a, y radius=\b];

  \draw (O) ++(0:\rinA)
    arc[start angle=0,  end angle=180, x radius=\rinA, y radius=\rinB];
  \draw (O) ++(180:\rinA)
    arc[start angle=180,end angle=360, x radius=\rinA, y radius=\rinB];

  % ===== Координаты (без рисования узлов) =====
  % A_i
  \foreach \i in {1,...,\NA}{
    \pgfmathsetmacro{\ang}{190 + (\i)*(150/(\NA+1))}
    \coordinate (A\i) at ($(T)+({\a*cos(\ang)},{\b*sin(\ang)})$);
  }
  % B_k на внешнем овале
  \foreach \k in {1,...,\Bfront}{
    \pgfmathsetmacro{\ang}{190 + (\k)*(150/(\TB+1))}
    \coordinate (B\k) at ($(O)+({\a*cos(\ang)},{\b*sin(\ang)})$);
  }
  % B_{j+3..j+5} внутри малого круга
  \pgfmathtruncatemacro{\kA}{\JJT}
  \pgfmathtruncatemacro{\kB}{\JJT+1}
  \pgfmathtruncatemacro{\kC}{\JJT+2}
  \coordinate (B\kA) at ($(O)+({\rinA*0.78*cos(\phiA)},{\rinB*0.78*sin(\phiA)})$);
  \coordinate (B\kB) at ($(O)+({\rinA*0.78*cos(\phiB)},{\rinB*0.78*sin(\phiB)})$);
  \coordinate (B\kC) at ($(O)+({\rinA*0.78*cos(\phiC)},{\rinB*0.78*sin(\phiC)})$);
  % C_m
  \foreach \m in {1,...,\JC}{
    \pgfmathsetmacro{\ang}{190 + (\m)*(150/(\JC+1))}
    \coordinate (C\m) at ($(U)+({\a*cos(\ang)},{\b*sin(\ang)})$);
  }

  % === УВОДИМ A1 и B1 внутрь и СДВИГАЕМ ВВЕРХ ===
  \coordinate (A1) at ($(T)!\inset!(A1)$);
  \coordinate (A1) at ($(A1)+(0,\dyA)$);
  \coordinate (B1) at ($(O)!\inset!(B1)$);
  \coordinate (B1) at ($(B1)+(0,\dyB)$);

  % ===== Боковые образующие =====
  % Прямой цилиндр
  \draw (-\a,0) -- (-\a,\h) (\a,0) -- (\a,\h);

  % Наклонный цилиндр: касания не на краях
  \coordinate (CtL) at ($(U)+({\a*cos(180+\phiTop)},{\b*sin(180+\phiTop)})$);
  \coordinate (CtR) at ($(U)+({\a*cos(360-\phiTop)},{\b*sin(360-\phiTop)})$);
  \coordinate (CbL) at ($(O)+({\a*cos(180+\phiBase)},{\b*sin(180+\phiBase)})$);
  \coordinate (CbR) at ($(O)+({\a*cos(360-\phiBase)},{\b*sin(360-\phiBase)})$);
  \draw (CbR) -- (CtR);  % задняя 
  \draw         (CbL) -- (CtL);  % передняя — сплошная

  % ===== Только нужные рёбра =====
  \draw[missing] (A1)--(B1);                 % отсутствует {a1,b1}
  \foreach \m in {1,...,\JC}{ \draw[missing] (B1)--(C\m); } % отсутствуют b1—C_m
  \draw (B1)--(B\JJT);                        % {b1,b_{j+3}} — сплошная

  % ===== ВЕРХНИЕ ОВАЛЫ A и C (в конце) =====
  \draw (U) ++(0:\a)
    arc[start angle=0,  end angle=180, x radius=\a, y radius=\b];
  \draw (U) ++(180:\a)
    arc[start angle=180,end angle=360, x radius=\a, y radius=\b];
  \draw (T) ++(0:\a)
    arc[start angle=0,  end angle=180, x radius=\a, y radius=\b];
  \draw (T) ++(180:\a)
    arc[start angle=180,end angle=360, x radius=\a, y radius=\b];

  % ===== Подписи (A и C над, B под) =====
  \node[font=\scriptsize, anchor=south] at ($(T)+(0,\b+0.20)$) {$A$};
  \node[font=\scriptsize, anchor=south] at ($(U)+(0,\b+0.20)$) {$C$};
  \node[font=\scriptsize, anchor=north] at ($(O)+(0,-\b-0.25)$) {$B$};

\end{tikzpicture}
\caption{The weakly $K_{s,t}$-saturated graph $G$: $A$ and the subset $\{b_{j+3},\ldots,b_t\}$ of $B$ induce cliques, all edges between sets $A,B$ and between sets $B,C$ are drawn except for the edges represented by dashed lines, the edge $\{b_1,b_{j+3}\}$ is represented by a solid line.}
\label{example.jpg}
\end{figure}
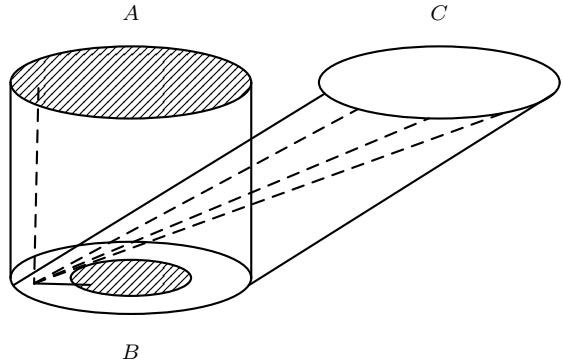

It is easy to see that $G$ has the required number of edges. Let us show that $G$ is weakly $K_{s,t}$-saturated, i.e. that the missing edges can be added one by one, each time creating a copy of $K_{s,t}$:
\begin{itemize}
    \item add the only missing edge $\{a_1,b_1\}$ between $A$ and $B$;
    %\item create \( A_1A_2 \) by fixing \( A_1, B_1A_3, \ldots, A_s \) and \( A_2B_2, \ldots, B_s, \)
    \item for every $i\in[j]$, add the edge $\{b_1,c_i\}$ and create a copy of $K_{s,t}$ with parts $\{c_i, a_2, \ldots, a_s\}$ and $B$;
    \item add every edge $\{a_i,c_{i'}\}$ between sets $A$ and $C$ and create a copy of $K_{s,t}$ with parts $\{c_{i'}, a_1, \ldots, a_{i-1},a_{i+1}, \ldots, a_s\}$ and $\{a_i, b_1, \ldots, b_{t-1}\}$;
    \item add every edge $\{c_i,c_{i'}\}$ inside $C$ --- it creates a copy of $K_{s,t}$ with parts $\{c_i, a_1, \ldots, a_{s-1}\}$ and $\{c_{i'},b_1, \ldots, b_{t-1}\}$;
    \item for every $i\in[j+1]$, add the edge $\{b_{1+i},b_{j+3}\}$ --- it creates a copy of $K_{s,t}$ with the $t$-part $\{a_1, \ldots, a_{j+1}, b_{j+4}, \ldots, b_t, b_1, b_{1+i}\}$ and the $s$-part $\{c_1, \ldots, c_j, a_{j+2}, \ldots, a_s, b_{j+3}\}$ when $j\leq s-2$ and $\{c_1, \ldots, c_{s-1}, b_{j+3}\}$ when $j\geq s-1$; 
%    first parts $\{c_1, \ldots, c_j, a_{j+2}, \ldots, a_s, b_{j+3}\}$ and $\{a_1, \ldots, a_{j+1}, b_{j+4}, \ldots, b_t, b_1, b_{1+i}\}$, [Do we require here that $j\leq s-1$ or something like that?]
    \item for every $1\leq i<j+3<i'\leq t$, add the edge $\{b_i,b_{i'}\}$ and create a copy of $K_{s,t}$ with parts $\{b_i', a_1, \ldots, a_{s-1}\}$  and $\{b_i, c_1,\ldots,c_j, b_1, a_s, b_{j+3}, \ldots, b_{i'-1}, b_{i'+1}, \ldots, b_t\}$;
    \item add every missing edge $\{b_i,b_{i'}\}$, where $1\leq i<i'\leq j+2$, --- it creates a copy of $K_{s,t}$ with parts $\{b_i,a_1,\ldots,a_{s-1}\}$ and $\{b_{i'},a_s,c_1,\ldots,c_j,b_{j+3},\ldots,b_t\}$.
\end{itemize}

\subsection{Lower bounds}
\label{sc:2_lower_proof}

Let $j\geq 2$ and $G$ be a graph on at least $j+1$ vertices. Let us assume that there exists an edge $e\in E(G)$ and two disjoint sets $V^1,V^2\subset E(G)$ such that $e$ is the only edge between $V^1$ and $V^2$ and $|V(G)|-|V^1|-|V^2|=j$. We will call {\it $j$-erase procedure} the operation of deletion of $e$ from $G$. The graph $G$ is {\it $j$-erasable}, if the $j$-erase procedure can be applied certain number of times so that the final graph is empty. For technical reasons, we assume that the $j$-erase procedure can be applied to any edge of any graph on $j+1$ vertices.

Let $G$ be a $j$-erasable graph on $n$ vertices with the maximum number of edges, and let $n=s+t+j$. Then, similarly to Claim~\ref{claim}, we get
$$
 |E(G)|\geq{n\choose 2}-\mathrm{wsat}(n,K_{s,t}).
$$
For $j+1\leq k\leq n$, let $f_j(k)$ be the maximum number of edges in an erasable graph on $n$ vertices with $n-k$ isolated vertices.

\subsubsection{Lower bound for $j=2$}

We prove here that, for every $j\geq 2$, $f_j(k)\leq \alpha n-\beta$, where
$$
 \alpha={j+1\choose 2}+1,\quad \beta=\alpha j+1,
$$
which immediately implies the lower bound in Theorem~\ref{result_3} for $j=2$.
 
\begin{claim}
\label{cl:breakdown}
\begin{enumerate}
\item For $k\leq j+2$, $f_j(k)=\binom{k}{2}$.
\item For $k\geq j+3$, $f_j(k)\leq 1+\max_{1+j\leq k_1,k_2\leq k-1:\,k_1+k_2=k+j}(f_j(k_1)+f_j(k_2))$
\end{enumerate}
\end{claim}
\begin{proof}
The first part is immediate. Let us prove the second part. 
Let $G$ be a $j$-erasable graph on $n$ vertices with $n-k$ isolated vertices and with the maximum number of edges $f_j(k)$. Let $U$ be the set of $k$ non-isolated vertices. Assume the $j$-erase procedure is applied to an edge $e$ of $G$ and a partition $(V^1,V^2)$ witnesses this procedure. Without loss of generality, $U^0:=V(G)\setminus(V^1\cup V^2)$ is entirely inside $U$. Let $U^1:=U^0\cup(V^1\cap U)$ and $U^2:=U^0\cup(V^2\cap U)$ have sizes $k_1$ and $k_2$, respectively. Observe that graphs $G^1$ and $G^2$ obtained from $G[U^1]$ and $G[U^2]$ by adding $n-k_1$ and $n-k_2$ isolated vertices respectively are erasable, completing the proof. Indeed, let $i\in\{1,2\}$. Consider a sequence of $j$-erase procedures applied in $G$ that halts at the empty graph. The edges in this sequence that are entirely inside $U^i$ can also be erased in $G^i$ via exactly the same partition.
\end{proof}

Let us prove $f_j(k)\leq\alpha k-\beta$ by induction over $k\in\{j+1,\ldots,n\}$. For $k\in\{j+1,j+2\}$, Claim~\ref{cl:breakdown} gives
\begin{align*}
 f_j(j+1)&=\frac{(j+1)j}{2}=\alpha-1=\alpha(j+1)-\beta,\\
 f_j(j+2)&=\frac{(j+2)(j+1)}{2}\leq 2\alpha-1=\alpha(j+2)-\beta,
\end{align*}
as needed. Assume $k\geq j+3$ and $f_j(k')\leq\alpha k'-\beta$ for all $k'\in\{j+1,\ldots,k-1\}$. By Claim~\ref{cl:breakdown}, we get that there exist $k_1,k_2\in\{j+1,\ldots,k-1\}$ such that $k_1+k_2=k+j$ and 
$$
 f_j(k)\leq 1+f_j(k_1)+f_j(k_2)=1+\alpha k_1-\beta+\alpha k_2-\beta=1+\alpha(k+j)-\beta-\alpha j-1=\alpha k-\beta,
$$
completing the proof.

\subsubsection{Lower bound for $3\leq j\leq \frac{2}{3}(s+t)-\frac{5}{3}$}

We have to prove that $f_j(k)\leq\frac{19}{12}(j+1)(s+t-1)$. Let $G$ be an $n$-vertex $j$-erasable graph.

\begin{claim}
\label{cl:connectivity_reduction}
The graph $G$ does not have a $(j+2)$-connected subgraph.
\end{claim}

\begin{proof}
Assume, towards the contradiction, there exists a $(j+2)$-connected subgraph $H\subset G$. Let $e$ be the first edge erased in $H$. Then, there exist disjoint sets $V^1,V^2$ such that $n-|V^1|-|V^2|=j$ and $e$ is the unique edge between $V^1,V^2$. Since $e$ is the first edge erased in $H$, after its deletion there are still $j+1$ vertex-disjoint paths between the end points of $e$. Therefore, deletion of $V(G)\setminus (V^1\cup V^2)$ should preserve at least one path --- a contradiction.
\end{proof}

It remains to apply the following result of Bernshteyn and Kostochka.

\begin{theorem}[Bernshteyn, Kostochka, 2016~\cite{BK}]
\label{th:BK}
Let $k\geq 2$. Then every graph with at least $\frac{5k}{2}$ vertices and at least $\frac{19k}{12}(|V(G)|-k)$ edges has a $(k+1)$-connected subgraph.
\end{theorem}

Indeed, Claim~\ref{cl:connectivity_reduction} and Theorem~\ref{th:BK} imply that $G$ has at most $\frac{19}{12}(j+1)(s+t-1)$ edges.

\section{Discussions and remaining questions}
\label{sc:further}

For $t+s+2\leq n\leq 3t-3$ and $3\leq s\leq t$, the exact value of $\mathrm{wsat}(n,K_{s,t})$ remains unknown. We propose the following conjecture.

\begin{conjecture}[Weak]
For all $2<s\leq t$ and $2\leq j\leq t-3$,
$$
 \mathrm{wsat}(s+t+j,K_{s,t})={s+t+j\choose 2}-j(s+t)-2t+O(j).
$$
\end{conjecture}

It seems plausible that the $O(j)$-term equals $2j+O(1)$. At least this is the case when $j\in\{1,2\}$. Therefore, we believe that, even if the Mader's conjecture on the value of $\varphi_n(k)$ is true for all $k$, it does not imply a tight lower bound on $\mathrm{wsat}(s+t+j,K_{s,t})$, as it happens in the case $j=1$. It is also of interest to further explore the case $j=1$ and $s\neq t$ to determine all pairs $(s,t)$ where 
\begin{equation}
\mathrm{wsat}(s+t+1,K_{s,t})={s+t+1\choose 2}-(2s+2t-2)
\label{eq:gcd-1}
\end{equation} 
as in the case $\gcd(s,t)=1$. Note that, as follows from the result of~\cite{Miralaei2023}, for $s=2$, the equality~\eqref{eq:gcd-1} hold if and only if $t$ is odd. Nevertheless, there are pairs $(s\geq 3,t>s)$ such that $\gcd(s,t)>1$ but~\eqref{eq:gcd-1} holds as well: we conducted computer simulations and found several constructions that improve the upper bound in Theorem~\ref{result_1} and give the answer~\eqref{eq:gcd-1} for $(s,t)\in\{(4,6),(6,8),(8,10)\}$. The respective erasable graphs are presented in Appendix~\ref{appendix}. 

%for small values of $s$ and $t$, with results summarised in Figure~\ref{Table} {\color{red} Remove figure, add constructions for $K_{4,6},K_{6,8},K_{8,10}$}, which might provide reference points for further investigation. 

More ambitiously, we propose the following for balanced bipartite graphs.

\begin{conjecture}[Strong]
For all $s>2$ and $2\leq j\leq s-3$,
$$
\mathrm{wsat}(2s+j,K_{s,s})={2s+j\choose 2}-(j+1)(2s-2).
$$
\end{conjecture}

As we mention in Introduction, we expect that the linear-algebraic approach does not give a tight lower bound on $\mathrm{wsat}(s+t+j,K_{s,t})$. More generally, we would like to ask the following question. Is it true that for every pair $(s,t)$ such that $3\leq s\leq t$ and for sufficiently small $j\geq 1$, the weak $K_{s,t}$-saturation rank $\mathrm{rk}\text{-}\mathrm{sat}(s+t+j,K_{s,t})$ is strictly smaller than the weak $K_{s,t}$-saturation number $\mathrm{wsat}(s+t+j,K_{s,t})$ (see the definition of the weak saturation rank and the relevant background on the theory of matroids in~\cite{TZ:matroids})?

We conclude this section with a research direction that lies somewhat outside the scope of the present paper and concerns the case where $s$ and $t$ are small compared to $n$. In~\cite{Kalinichenko}, it is proved that, for every constant $p\in(0,1)$, the weak $K_{s,t}$-saturation number is stable with high probability (whp). More formally, for every $1\leq s\leq t$, whp $\mathrm{wsat}(G_{n,p},K_{s,t})=\mathrm{wsat}(n,K_{s,t})$, where $G_{n,p}$ is a binomial random graph with edge probability $p$. Here, $\mathrm{wsat}(H,F)$ is a generalisation of $\mathrm{wsat}(n,F)$, where the added edges must belong to $E(H)$.  Kor\'{a}ndi and Sudakov~\cite{KorSud} initiated the study of weak saturation in random graphs: they proved the stability property for $F=K_s$ and asked to determine the threshold for the stability property. In~\cite{CPSTZ}, it was proved that the threshold for a triangle $F\cong K_3$ equals $n^{-1/3+o(1)}$ via a topological approach. Some non-trivial bounds on the threshold for $F\cong K_s$, $s\geq 4$, were obtained in~\cite{BMTZ}. The threshold for stars $F\cong K_{s=1,t}$ was determined in~\cite{Kalinichenko}. For all the other $s\geq 2$, no non-trivial bounds on the threshold are known. The next interesting case to investigate is a 4-cycle $F\cong K_{2,2}$.

\begin{comment}
    \begin{cases}
\binom{s+t+1}{2}-(2s+2t-2) &\text{for } (s,t)=1, \\
 \binom{s+t+1}{2}-(2s+2t-2)~~or~~ \binom{s+t+1}{2}-(2s+2t-3) &\text{for } (s,t)\neq 1.
\end{cases}
\end{comment}

%\renewcommand{\refname}{References}
\bibliographystyle{abbrv}
\bibliography{references}

\begin{appendix}
\section{Examples of erasable graphs}
\label{appendix}

\begin{itemize}

\item $s=4$, $t=6$. An erasable graph on 11 vertices and 18 edges is defined as follows. Let $P_a=(a_1a_2a_3a_4)$ and $P_b=(b_1b_2b_3b_4)$ be two vertex-disjoint paths and let vertices $a$ and $b$ be adjacent to all vertices in $P_a$ and $P_b$, respectively. It remains to add a vertex $c$ that is adjacent to $a_1,a_4,b_4$, and draw the edge $\{a_4,b_2\}$.

\item $s=6$, $t=8$. An erasable graph on $\{1,\ldots,15\}$ has the following set of 26 edges: $\{1,6\}$, $\{1,8\}$, $\{2,10\}$, $\{2,12\}$, $\{3,4\}$, $\{3,5\}$, $\{3,13\}$, $\{3,14\}$, $\{4,7\}$, $\{4,12\}$, $\{4,14\}$, $\{5,7\}$, $\{5,9\}$, $\{5,10\}$, $\{6,8\}$, $\{6,10\}$, $\{6,12\}$, $\{7,9\}$, $\{7,14\}$, $\{8,10\}$, $\{8,15\}$, $\{10,12\}$, $\{10,15\}$, $\{11,13\}$, $\{12,13\}$, $\{13,14\}$.

\item $s=8$, $t=10$. An erasable graph on $\{1,\ldots,19\}$ has the following set of 34 edges: $\{1,3\}$, $\{1,6\}$, $\{1,9\}$, $\{1,11\}$, $\{1,16\}$, $\{1,17\}$, $\{1,18\}$, $\{2,7\}$, $\{2,9\}$, $\{2,10\}$, $\{2,14\}$, $\{3,8\}$, $\{3,11\}$, $\{3,18\}$, $\{4,9\}$, $\{4,13\}$, $\{4,14\}$, $\{4,15\}$, $\{5,6\}$, $\{5,8\}$, $\{6,11\}$, $\{7,13\}$, $\{8,17\}$, $\{8,19\}$, $\{9,14\}$, $\{9,15\}$, $\{9,19\}$, $\{10,13\}$, $\{10,17\}$, $\{11,17\}$, $\{12,14\}$, $\{12,15\}$, $\{13,14\}$, $\{18,19\}$.

\end{itemize}

\end{appendix}

\begin{comment}

\end{comment}

\end{document}